\documentclass{amsart}
\usepackage{graphicx, amssymb}
\usepackage[alphabetic,y2k,lite]{amsrefs}
\usepackage{tikz}
\usetikzlibrary{shapes.geometric}
\usetikzlibrary{calc}
\newcommand{\tikzsetup}{
  \tikzstyle tikzdefault=[semithick,>=stealth];
  \tikzstyle dotnode=[fill,circle,inner sep=1pt];
  \tikzstyle morphism=[circle,draw,inner sep=1pt, minimum
size=18pt, scale=0.8];
  \tikzstyle coupon=[draw,inner sep=1pt, minimum size=18pt,scale=0.8];
  \tikzstyle overline=[preaction={draw,line
width=2mm,white,-},semithick,>=stealth];
  \tikzstyle drinfeld center=[>=stealth,green!60!black, double
distance=1pt]];
  \tikzstyle cell=[fill=black!10,draw, very thin];
   \tikzstyle midarrow=[fill, dart, dart tail angle=112,
    dart tip angle=53,inner sep=0.55pt,sloped,allow upside down];
   \tikzstyle revmidarrow=[fill, dart, dart tail angle=112,
    dart tip angle=53,inner sep=0.55pt,sloped,
    allow upside down,rotate=180];
}
\newcommand{\OrientedCircle}[1][]{
  \begin{scope}[tikzdefault,#1]
  \draw[->] (-1,0) arc (-180:0:1);
  \draw[-] (1,0) arc (0:180:1);
  \end{scope}
}

\newcommand{\HalfDashedOrientedCircle}[1][]{
  \begin{scope}[tikzdefault,#1]
 \draw[densely dashed, -] (0,1) arc (90:270:1);
  \draw[->] (0,-1) arc (-90:0:1); 
  \draw[-] (1,0) arc (0:90:1);
   \end{scope}
}
\newcommand{\tzRoundCouponI}{
\begin{tikzpicture}[baseline=0pt]
\tikzsetup
\node[morphism] (ph) at (0,0) {$\ph$};
\draw[tikzdefault,->] (ph)-- +(240:1cm) node[pos=0.7, left] {$V_1$} ;
\draw[tikzdefault,->] (ph)-- +(180:1cm);
\draw[tikzdefault,->] (ph)-- +(120:1cm);
\draw[tikzdefault,->] (ph)-- +(60:1cm);
\draw[tikzdefault,->] (ph)-- +(0:1cm);
\draw[tikzdefault,->] (ph)-- +(-60:1cm) node[pos=0.7, right] {$V_n$};
\end{tikzpicture}
}

\newcommand{\tzRoundCouponII}{
\begin{tikzpicture}[baseline=0pt]
\tikzsetup
\node[coupon] (ph) at (0,0) {$\ph$};
\draw[tikzdefault,->] (ph.130).. controls +(130:0.5cm) and +(90:1cm) ..
                  (-1,-1) node[pos=0.9, left] {$V_1$} ;
\draw[tikzdefault,->] (ph.115).. controls +(90:1cm) and +(0:0.5cm) ..
                  (-1.5,0) ;
\draw[tikzdefault,->] (ph.100) -- +(0,1cm);
\draw[tikzdefault,->] (ph.80) -- +(0,1cm);
\draw[tikzdefault,->] (ph.65).. controls +(90:1cm) and +(180:0.5cm) ..
                  (1.5,0) ;
\draw[tikzdefault,->] (ph.50).. controls +(50:0.5cm) and +(90:1cm) ..
                  (1,-1) node[pos=0.9, right] {$V_n$} ;
\end{tikzpicture}
}
\newcommand{\tzSummationConventionI}{
\begin{tikzpicture}[baseline=0pt]
\tikzsetup
\node[morphism] (ph) at (0,0) {$\ph$};
\draw[tikzdefault,->] (ph)-- +(240:1cm) node[pos=0.7, left] {$V_1$} ;
\draw[tikzdefault,->] (ph)-- +(180:1cm);
\draw[tikzdefault,->] (ph)-- +(120:1cm);
\draw[tikzdefault,->] (ph)-- +(60:1cm);
\draw[tikzdefault,->] (ph)-- +(0:1cm);
\draw[tikzdefault,->] (ph)-- +(-60:1cm) node[pos=0.7, right] {$V_n$};
\node[morphism] (ph') at (3,0) {$\ph^*$};
\draw[tikzdefault,->] (ph')-- +(240:1cm) node[pos=0.7, left] {$V_n^*$} ;
\draw[tikzdefault,->] (ph')-- +(180:1cm);
\draw[tikzdefault,->] (ph')-- +(120:1cm);
\draw[tikzdefault,->] (ph')-- +(60:1cm);
\draw[tikzdefault,->] (ph')-- +(0:1cm);
\draw[tikzdefault,->] (ph')-- +(-60:1cm) node[pos=0.7, right] {$V_1^*$};
\end{tikzpicture}
}
\newcommand{\tzSummationConventionII}{
\begin{tikzpicture}[baseline=0pt]
\tikzsetup
\node[morphism] (ph) at (0,0) {$\ph_\al$};
\draw[tikzdefault,->] (ph)-- +(240:1cm) node[pos=0.7, left] {$V_1$} ;
\draw[tikzdefault,->] (ph)-- +(180:1cm);
\draw[tikzdefault,->] (ph)-- +(120:1cm);
\draw[tikzdefault,->] (ph)-- +(60:1cm);
\draw[tikzdefault,->] (ph)-- +(0:1cm);
\draw[tikzdefault,->] (ph)-- +(-60:1cm) node[pos=0.7, right] {$V_n$};
\node[morphism] (ph') at (3,0) {$\ph^\al$};
\draw[tikzdefault,->] (ph')-- +(240:1cm) node[pos=0.7, left] {$V_n^*$} ;
\draw[tikzdefault,->] (ph')-- +(180:1cm);
\draw[tikzdefault,->] (ph')-- +(120:1cm);
\draw[tikzdefault,->] (ph')-- +(60:1cm);
\draw[tikzdefault,->] (ph')-- +(0:1cm);
\draw[tikzdefault,->] (ph')-- +(-60:1cm) node[pos=0.7, right] {$V_1^*$};
\end{tikzpicture}
}
\newcommand{\tzPairingIII}{
\begin{tikzpicture}[baseline=0pt]
\tikzsetup
\node[morphism] (ph) at (0,0) {$\ph$};
\node[morphism] (ph*) at (2,0) {$\ph^*$};
\node at (0,0.7) {$\dots$};
\node at (2,0.7) {$\dots$};
\draw[tikzdefault,->] (ph)-- +(120:1cm) node[pos=1.0,above,scale=0.8]
{$V_1$};
\draw[tikzdefault,->] (ph)-- +(60:1cm) node[pos=1.0,above,scale=0.8]
{$V_n$};
\draw[tikzdefault,<-] (ph*)-- +(120:1cm) node[pos=1.0,above,scale=0.8]
{$V_n$};
\draw[tikzdefault,<-] (ph*)-- +(60:1cm) node[pos=1.0,above,scale=0.8]
{$V_1$};
\draw[tikzdefault,->] (ph) -- (ph*) node[pos=0.5,below,scale=0.8] {$X_i$}; 
\end{tikzpicture}
}
\newcommand{\tzPairingIV}{
\begin{tikzpicture}[baseline=-1cm]
\tikzsetup
\node at (0.2,0) {$\dots$};
\node at (2.7,0) {$\dots$};
\draw[tikzdefault,<-] (0,0) arc(190:350:1.5cm);
 \node at(1.5,-1.5) {$V_1$};
\draw[tikzdefault,<-] (0.5,0.2) arc(190:350:1cm);
 \node at(1.5,-0.3) {$V_n$};
\end{tikzpicture}
}
\newcommand{\tzPairingV}{
\begin{tikzpicture}[baseline=0pt]
\tikzsetup
\node[circle, draw, dashed, minimum size=1.2cm,label=135:$A$] (A) at (0,0)
{} ;
\node[morphism] (ph) at (1.5,0) {$\ph$};
\node[morphism] (ph') at (2.5,0) {$\ph^*$};
\node[circle, draw, dashed, minimum size=1.2cm, label=45:$B$] (B) at (4,0)
{};
\draw[tikzdefault, ->, out=45,in=135] (A) to (ph) 
                node[pos=0.5,above] {$V_1$};
\draw[tikzdefault, ->,out=15,in=165] (A) to (ph);
\draw[tikzdefault, ->,out=-15,in=195] (A) to (ph);
\draw[tikzdefault, ->,out=-45,in=225] (A) to (ph) 
                node[pos=0.5,below] {$V_n$};
\draw[tikzdefault, ->,out=45,in=135] (ph') to (B) 
                node[pos=0.5,above] {$V_1$};
\draw[tikzdefault, ->,out=15,in=165] (ph') to (B);
\draw[tikzdefault, ->,out=-15,in=195] (ph') to (B);
\draw[tikzdefault, ->,out=-45,in=225] (ph') to (B) 
                node[pos=0.5,below] {$V_n$};
\end{tikzpicture}
}
\newcommand{\tzPairingVI}{
\begin{tikzpicture}[baseline=0pt]
\tikzsetup
\node[circle, draw, dashed, minimum size=1.2cm,label=135:$A$] (A) at (0,0)
{};
\node[circle, draw, dashed, minimum size=1.2cm,label=45:$B$] (B) at (3,0)
{};
\draw[tikzdefault, ->,out=45,in=135] (A) to (B) 
                node[pos=0.5,above] {$V_1$};
\draw[tikzdefault, ->,out=15,in=165] (A) to (B);
\draw[tikzdefault, ->,out=-15,in=195] (A) to (B);
\draw[tikzdefault, ->,out=-45,in=225] (A) to (B) 
                node[pos=0.5,below] {$V_n$};
\end{tikzpicture}
}
\newcommand{\tzPairingVII}{
\begin{tikzpicture}[baseline=0pt]
\tikzsetup
\node[morphism] (ph) at (0,0) {$\ph$};
\node[morphism] (ph*) at (2,0) {$\ph^*$};
\node at (0,0.7) {$\dots$};
\node at (2,0.7) {$\dots$};
\draw[tikzdefault,->] (ph)-- +(120:1cm) node[pos=1.0,above,scale=0.8]
{$V_1$};
\draw[tikzdefault,->] (ph)-- +(60:1cm) node[pos=1.0,above,scale=0.8]
{$V_n$};
\draw[tikzdefault,<-] (ph*)-- +(120:1cm) node[pos=1.0,above,scale=0.8]
{$V_n$};
\draw[tikzdefault,<-] (ph*)-- +(60:1cm) node[pos=1.0,above,scale=0.8]
{$V_1$};
\draw[tikzdefault,out=15,in=165] (ph) to (ph*)
node[pos=0.5,above,scale=0.8] {$X_j$};
\draw[tikzdefault,out=-15,in=195] (ph) to (ph*)
node[pos=0.5,below,scale=0.8] {$X_k$};
\end{tikzpicture}
}

\newcommand{\tzPairingVIII}{
\begin{tikzpicture}[baseline=0pt]
\tikzsetup
\node[morphism] (ph) at (0,0) {$\ph$};
\node[morphism] (ph*) at (2,0) {$\ph^*$};
\draw[tikzdefault,->] (ph)-- +(90:1cm) node[pos=0.7,left,scale=0.8] {$V$};
\draw[tikzdefault,<-] (ph*)--  +(90:1cm) node[pos=0.7,left,scale=0.8]
{$V$};
\draw[tikzdefault,out=15,in=165] (ph) to (ph*)
node[pos=0.5,above,scale=0.8] {$X_j$};
\draw[tikzdefault,out=-15,in=195] (ph) to (ph*)
node[pos=0.5,below,scale=0.8]
{$X_k$};
\end{tikzpicture}
}

\newcommand{\tzPairingIX}{
\begin{tikzpicture}[baseline=0pt]
\tikzsetup
\node[morphism] (ph) at (0,0) {$\ph$};
\node[morphism] (ph*) at (2,0) {$\ph^*$};
\draw[tikzdefault,->] (ph)-- +(90:1cm) node[pos=0.7,left,scale=0.8] {$V$};
\draw[tikzdefault,<-] (ph*)--  +(90:1cm) node[pos=0.7,left,scale=0.8]
{$V$};
\draw[tikzdefault,->] (ph) -- (ph*) node[pos=0.5,below,scale=0.8] {$X_i$};
\end{tikzpicture}
}

\newcommand{\tzSphereDiagram}{
 \begin{tikzpicture}[baseline=0pt]
\tikzsetup
\node[morphism] (C) at (0,1.7cm) {$C$};
\node[morphism] (A) at (-1.2cm,0.7cm) {$A$};
\node[morphism] (B) at (1.2cm,0.7cm) {$B$};
\node[morphism] (D) at (-1.2cm,-0.7cm) {$D$};
\node[morphism] (E) at (1.2cm,-0.7cm) {$E$};
\node[morphism] (F) at (0,-1.7cm) {$F$};
\begin{scope}[tikzdefault,->]
\draw (C)--(A);
\draw (A) -- (B);
\draw (B) -- (C);
\draw (D)--(A);
\draw (E)--(B);
\draw (D)--(E);
\draw (D)--(F);
\draw (E)--(F);
\draw[<-] (C.45) .. controls +(45:1cm) and +(90:2cm) .. (2,0)
.. controls +(-90:2cm) and +(-45:1cm) .. (F);
\end{scope}
\end{tikzpicture}
}
\newcommand{\tzCrossing}{
\begin{tikzpicture}[baseline=0pt]
\tikzsetup
\draw[drinfeld center](0,0)--(1cm,2cm) node [pos=0.2, left,black] {$Y$};
\draw[overline](1cm,0) -- (0,2cm) node [pos=0.2, right] {$V$};
\end{tikzpicture}
}

\newcommand{\tzProjectionI}{
\begin{tikzpicture}[baseline=0pt]
\tikzsetup
\node[coupon,scale=0.8] (ph) at (0,0) {$\psi$};
\draw[drinfeld center](0,-1.5)--(ph) node[pos=0.2, right, black,scale=0.8]
{$Y$};
\draw[drinfeld center](ph) -- (0,1.5) node[pos=0.9, right,black, scale=0.8]
{$Z$};
  \draw[->,overline] (-0.8,0) arc (-180:0:0.8);
  \draw[-,overline] (0.8,0) arc (0:180:0.8);
\node at (135:1cm) {$X$};
\end{tikzpicture}
}
\newcommand{\tzProjectionIa}{
\begin{tikzpicture}[baseline=0pt]
\tikzsetup
\node[coupon,scale=0.8] (ph) at (0,0) {$\Phi$};
\draw[drinfeld center](0,-1.5)--(ph) node[pos=0.2, right, black,scale=0.8]
{$Z$};
\draw[drinfeld center](ph) -- (0,1.5) node[pos=0.9, right,black, scale=0.8]
{$Z$};
  \draw[->,overline] (-0.8,0) arc (-180:0:0.8);
  \draw[-,overline] (0.8,0) arc (0:180:0.8);
\node at (135:1cm) {$j$};
\end{tikzpicture}
}

\newcommand{\tzProjectionII}{
\begin{tikzpicture}[baseline=0pt]
\tikzsetup
\node[coupon,scale=0.8] (ph) at (0,0) {$\psi$};
\draw[drinfeld center](0,-1.8)--(ph) node[pos=0.2, right, black,scale=0.8]
{$Y$};
\draw[drinfeld center](ph) -- (0,1.8) node[pos=0.9, right,black, scale=0.8]
{$Z$};
  \draw[->,overline] (-0.8,0) arc (-180:0:0.8);
  \draw[-,overline] (0.8,0) arc (0:180:0.8);
\node at (135:1cm) {$j$};
\draw[overline] (1.2,-1.8) .. controls +(0cm,1cm) and +(0cm,-1cm) ..
(-1.2,-0.5)-- (-1.2,1.8) 
   node[pos=0.9,left, scale=0.8] {$W$};
\end{tikzpicture}
}

\newcommand{\tzProjectionIII}{
\begin{tikzpicture}[baseline=0pt]
\tikzsetup
\node[coupon,scale=0.8] (psi) at (0,0) {$\psi$};
\draw[drinfeld center](0,-1.8)--(psi) node[pos=0.2, right, black,scale=0.8]
{$Y$};
\draw[drinfeld center](psi) -- (0,1.8) node[pos=0.9, right,black,
scale=0.8]
{$Z$};
\draw[->,overline] (0,0.8) arc (90:270:0.8);
\draw[->,overline] (0,-0.8) arc (-90:90:0.8);
\node[morphism, scale=0.8,fill=white] (ph) at (-0.8,0) {$\ph$};  
\node[morphism, scale=0.8,fill=white] (ph*) at (0.8,0) {$\ph^*$};  
\node at (135:1cm) {$j$};
\node at (-45:1cm) {$i$};
\draw[tikzdefault,out=120,in=-90] (ph.120) to (-1.2,1.8)
   node[pos=0.9,left, scale=0.8] {$W$};
\draw[tikzdefault,out=-60,in=90] (ph*.-60) to (1.2,-1.8)
   node[pos=0.9,right, scale=0.8] {$W$};
\end{tikzpicture}
}

\newcommand{\tzProjectionIV}{
\begin{tikzpicture}[baseline=0pt]
\tikzsetup
\node[coupon,scale=0.8] (ph) at (0,0) {$\psi$};
\draw[drinfeld center](0,-1.8)--(ph) node[pos=0.2, right, black,scale=0.8]
{$Y$};
\draw[drinfeld center](ph) -- (0,1.8) node[pos=0.9, right,black, scale=0.8]
{$Z$};
  \draw[->,overline] (-0.8,0) arc (-180:0:0.8);
  \draw[-,overline] (0.8,0) arc (0:180:0.8);
\node at (135:1cm) {$i$};
\draw[overline] (1.2,-1.8) -- (1.2,0.5) 
 .. controls +(0cm,1cm) and +(0cm,-1cm) .. (-1.2,1.8) 
   node[pos=0.9,left, scale=0.8] {$W$};
\end{tikzpicture}
}

\newcommand{\tzHalfBraid}{
\begin{tikzpicture}[baseline=0pt]
\tikzsetup
\coordinate (bot) at (0,-1.8);
\coordinate (top) at (0,1.8);
\draw [tikzdefault] (bot)--(top) node[pos=0.9, right] {$V$};
\node[morphism] (ph) at (-0.7,0) {$\ph$};
\node[morphism] (ph*) at (0.7,0) {$\ph^*$};
\draw[tikzdefault,->] (ph|-bot) -- (ph) node[pos=0.1,right] {$i$};
\draw[tikzdefault,->] (ph) -- (ph|-top) node[pos=0.8,right] {$j$};
\draw[tikzdefault, out=120, in=-90] (ph.120) to (-1.2,1.8)
node[pos=0.8,left] {$W$};
\draw[tikzdefault,<-] (ph*|-bot) -- (ph*) node[pos=0.1,right] {$i$};
\draw[tikzdefault,<-] (ph*) -- (ph*|-top) node[pos=0.8,right] {$j$};
\draw[tikzdefault, out=-60, in=90] (ph*.-60) to (1.2,-1.8)
     node[pos=0.8,right] {$W$};
\end{tikzpicture}
}

\newcommand{\tzFunctI}{
\begin{tikzpicture}[baseline=0pt]
\tikzsetup
\node[coupon,scale=0.8] (psi) at (0,0) {$\Psi$};
\draw[tikzdefault](0,-1.5)--(psi) node[pos=0.1, right, scale=0.8]
{$V$};
\draw[tikzdefault](psi.-60)-- +(0,-0.7cm) node[pos=0.8, right, scale=0.8]
{$\one$};
\draw[tikzdefault](psi.-120)-- +(0,-0.7cm) node[pos=0.8, left, scale=0.8]
{$\one$};
\draw[drinfeld center](psi) -- (0,1.2) node[pos=0.9, right,black,
scale=0.8] {$Z$};
\end{tikzpicture}
}

\newcommand{\tzFunctII}{
\begin{tikzpicture}[baseline=0pt]
\tikzsetup
\coordinate (V) at (0,-1);
\node[coupon,scale=0.8] (psi) at (0,0) {$\Phi$};
\draw[tikzdefault](V)--(psi) node[pos=0.1, right, scale=0.8]
{$V$};
\draw[drinfeld center](psi) -- (0,1.2) node[pos=0.9, right,black,
scale=0.8] {$Z$};
\draw[overline,->] (-0.7,-1)-- +(0,0.9) node[pos=0.1, left] {$i$} arc
(180:0:0.7cm) -- (0.7,-1);
\end{tikzpicture}
}
\newcommand{\tzFunctIII}{
\begin{tikzpicture}[baseline=0pt]
\tikzsetup
\coordinate (V) at (0,-1.5);
\node[coupon,scale=0.8] (psi) at (0,0) {$\Psi$};
\draw[tikzdefault](V)--(psi) node[pos=0.1, right, scale=0.8]
{$V$};
\draw[drinfeld center](psi) -- (0,1.2) node[pos=0.9, right,black,
scale=0.8] {$Z$};
\draw[overline,->] (-0.7,-1.5)-- +(0,1.4) node[pos=0.1, left] {$i$} arc
(180:0:0.7cm) -- (0.7,-1.5);
\draw[tikzdefault](psi.-60)-- +(0,-0.7cm)
node[pos=0.8, right, scale=0.8]
{$\one$};
\draw[tikzdefault](psi.-120)-- +(0,-0.7cm) node[pos=0.8, left, scale=0.8]
{$\one$};

\end{tikzpicture}
}

\newcommand{\tzFunctIV}{
\begin{tikzpicture}[baseline=0pt]
\tikzsetup
\coordinate (V) at (0,-1.5);
\node[coupon,scale=0.8] (psi) at (0,0) {$\Psi$};
\node[morphism,scale=0.8] (ph) at (-0.3,-0.7) {$\ph$};
\node[morphism,scale=0.8] (ph*) at (0.3,-0.7) {$\ph*$};
\draw[tikzdefault](V)--(psi) node[pos=0, right, scale=0.8]
{$V$};
\draw[drinfeld center](psi) -- (0,1.2) node[pos=0.9, right,black,
scale=0.8] {$Z$};
\draw[tikzdefault,->, out=90,in=180] (-0.8,-1.5) to (ph.180) node[pos=0.1,
left] {$i$};
\draw[tikzdefault,<-, out=90,in=0] (0.8,-1.5) to (ph*.0) node[pos=0.1,
right] {$i$};
\draw[tikzdefault](psi.-60)-- (ph*) node[pos=0.5,right=2pt] {$j^*$}
  -- +(0,-0.5) node[pos=0.8, right, scale=0.8] {$\one$};
\draw[tikzdefault](psi.-120)-- (ph) node[pos=0.5,left=2pt] {$j$} 
  -- +(0,-0.5) node[pos=0.8, left, scale=0.8] {$\one$};
\end{tikzpicture}
}

\newcommand{\tzFunctV}{
\begin{tikzpicture}[baseline=0pt]
\tikzsetup
\coordinate (V) at (0,-1.5);
\node[coupon,scale=0.8] (psi) at (0,0) {$\Psi$};
\draw[tikzdefault](V)--(psi) node[pos=0.1, right, scale=0.8]
{$V$};
\draw[drinfeld center](psi) -- (0,1.2) node[pos=0.9, right,black,
scale=0.8] {$Z$};
\end{tikzpicture}
}

\newcommand{\tzSphereDiagramII}{
 \begin{tikzpicture}[baseline=0pt]
\tikzsetup
\node[morphism] (C) at (0,1.7cm) {$C$};
\node[morphism] (A) at (-1.2cm,0.7cm) {$A$};
\node[morphism] (B) at (1.2cm,0.7cm) {$B$};
\node[morphism] (D) at (-1.2cm,-0.7cm) {$D$};
\node[morphism] (E) at (1.2cm,-0.7cm) {$E$};
\node[morphism] (F) at (0,-1.7cm) {$F$};
\begin{scope}[tikzdefault,->]
\draw (C)--(A);
\draw (A) -- (B);
\draw (B) -- (C);
\draw (E)--(B);
\draw (D)--(E);
\draw (D)--(F);
\draw (E)--(F);
\draw[<-] (C.45) .. controls +(45:1cm) and +(90:2cm) .. (2,0)
.. controls +(-90:2cm) and +(-45:1cm) .. (F);
\end{scope}
\draw[drinfeld center]
   (C.180) .. controls +(180:1.5cm) and +(150:1cm) ..
   (-1.2cm,0) .. controls +(-30:0.5cm) .. (E.150);
\draw[overline,->] (D)--(A);
\end{tikzpicture}
}

\newcommand{\tzHomotopyI}{
\begin{tikzpicture}[baseline=0pt]
\tikzsetup
\node[morphism] (ph) at (0,0) {$\ph$};
\node at (0,1.2) {$\dots$};
\node at (0,-1.2) {$\dots$};
\draw[drinfeld center] (-1,1)--(1,0.5);
\draw[overline] (ph)-- +(120:1.5cm);
\draw[overline] (ph)-- +(60:1.5cm);
\draw[tikzdefault] (ph)-- +(-120:1.5cm);
\draw[tikzdefault] (ph)-- +(-60:1.5cm);
\end{tikzpicture}
}

\newcommand{\tzHomotopyII}{
\begin{tikzpicture}[baseline=0pt]
\tikzsetup
\node[morphism] (ph) at (0,0) {$\ph$};
\node at (0,1.2) {$\dots$};
\node at (0,-1.2) {$\dots$};
\draw[drinfeld center] (-1,-0.5)--(1,-1);
\draw[tikzdefault] (ph)-- +(120:1.5cm);
\draw[tikzdefault] (ph)-- +(60:1.5cm);
\draw[overline] (ph)-- +(-120:1.5cm);
\draw[overline] (ph)-- +(-60:1.5cm);
\end{tikzpicture}
}

\newcommand{\tzStateSpaceI}{
   \begin{tikzpicture}[baseline=0pt]
   \tikzsetup
   \coordinate (a0) at (0:1.2cm);
   \coordinate (a1) at (72:1.2cm);
   \coordinate (a2) at (144:1.2cm);
   \coordinate (a3) at (216:1.2cm);
   \coordinate (a4) at (288:1.2cm);
  \fill[cell] (a0)--(a1)--(a2)--(a3)--(a4)--(a0);
  \draw[tikzdefault,->] (a0)--(a1) node [pos=0.5,auto=right] {$X_1$};
  \draw[tikzdefault,->] (a1)--(a2) node [pos=0.5,auto=right] {$X_2$};
  \draw[tikzdefault,->] (a2)--(a3) node [pos=0.5,auto=right] {$X_3$};
  \draw[tikzdefault,->] (a3)--(a4) node [pos=0.5,auto=right] {$X_4$};
  \draw[tikzdefault,->] (a4)--(a0) node [pos=0.5,auto=right] {$X_5$};
   \end{tikzpicture}
}

\newcommand{\tzStateSpaceII}{
   \begin{tikzpicture}[baseline=0pt]
   \tikzsetup
   \coordinate (a0) at (0:1.5cm);
   \coordinate (a1) at (72:1.5cm);
   \coordinate (a2) at (144:1.5cm);
   \coordinate (a3) at (216:1.5cm);
   \coordinate (a4) at (288:1.5cm);
  \fill[cell] (a0)--(a1)--(a2)--(a3)--(a4)--(a0);
  \draw[tikzdefault,->] (a0)--(a1) node [pos=0.9,right] {$X_1$};
  \draw[tikzdefault,->] (a1)--(a2) node [pos=0.8,above] {$X_2$};
  \draw[tikzdefault,->] (a2)--(a3) node [pos=0.8,left] {$X_3$};
  \draw[tikzdefault,->] (a3)--(a4) node [pos=0.8,below] {$X_4$};
  \draw[tikzdefault,->] (a4)--(a0) node [pos=0.8,right] {$X_5$};
  \draw[tikzdefault,red] (0,0) -- node[midarrow,pos=0.8]{} (36:3cm)
         node [pos=0.8,auto=right] {$X_1^*$};
  \draw[tikzdefault,red] (0,0) -- node[revmidarrow,pos=0.8]{} (108:3cm)
         node [pos=0.8,auto=right] {$X_2$};
  \draw[tikzdefault,red] (0,0) -- node[revmidarrow,pos=0.8]{} (180:3cm)
         node [pos=0.8,auto=right] {$X_3$};
  \draw[tikzdefault,red] (0,0) -- node[midarrow,pos=0.8]{} (252:3cm)
         node [pos=0.8,auto=right] {$X_4^*$};
  \draw[tikzdefault,red] (0,0) -- node[revmidarrow,pos=0.8]{} (324:3cm)
         node [pos=0.8,auto=right] {$X_5$};
   \end{tikzpicture}\qquad
   \begin{tikzpicture}[baseline=0pt]
   \tikzsetup
   \node[morphism] (ph) at (0,0) {$\ph_C$};
  \draw[tikzdefault] (ph) -- node[midarrow,pos=0.8]{} (36:1.5cm)
         node [pos=0.8,auto=right] {$X_1^*$};
  \draw[tikzdefault] (ph) -- node[revmidarrow,pos=0.8]{} (108:1.5cm)
         node [pos=0.8,auto=right] {$X_2$};
  \draw[tikzdefault] (ph) -- node[revmidarrow,pos=0.8]{} (180:1.5cm)
         node [pos=0.8,auto=right] {$X_3$};
  \draw[tikzdefault] (ph) -- node[midarrow,pos=0.8]{} (252:1.5cm)
         node [pos=0.8,auto=right] {$X_4^*$};
  \draw[tikzdefault] (ph) -- node[revmidarrow,pos=0.8]{} (324:1.5cm)
         node [pos=0.8,auto=right] {$X_5$};
\end{tikzpicture}
}

\newcommand{\tzGluingAxiomI}{
\begin{tikzpicture}[baseline=0pt]
\tikzsetup
\node[morphism] (ph) at (0,0) {$\ph$};
\node[morphism] (psi) at (2,0) {$\psi$};
\draw[tikzdefault] (ph) -- +(0,1cm) node[pos=0.8, left] {$A$};
\draw[tikzdefault] (psi) -- +(0,1cm) node[pos=0.8, right] {$B$};
\draw[drinfeld center] (ph)--(psi) node[pos=0.3,above,black]{$Z$};
\draw[->,overline](3,1)--(3,0) arc (0:-180:1cm) -- +(0,1)
node[pos=0.7,right] {$X$};
\end{tikzpicture}
}

\newcommand{\tzSphereProofI}{
\begin{tikzpicture}[baseline=0cm]
\tikzsetup
\path[use as bounding box] (-1,-1.5) rectangle (1,2.2);
\node[morphism] (phi) at (0,0) {$\ph$};
\draw[tikzdefault] (phi)-- +(0,1) node[pos=0.6,right=-1pt] {$\one$};
\draw[tikzdefault] (phi)-- +(0,-1) node[pos=0.6,right=-1pt] {$\one$};
\draw[tikzdefault,->] (phi).. controls +(140:3.5cm) and +(40:3.5cm) ..
(phi) node[pos=0.5,above] {$X$};
\end{tikzpicture}
}
\newcommand{\tzSphereProofII}{
\begin{tikzpicture}[baseline=0cm]
\tikzsetup
\path[use as bounding box] (-1,-1.5) rectangle (1,2.2);
\node[morphism] (phi) at (0,0) {$\ph^*$};
\draw[tikzdefault] (phi)-- +(0,1) node[pos=0.6,right=-1pt] {$\one$};
\draw[tikzdefault] (phi)-- +(0,-1) node[pos=0.6,right=-1pt] {$\one$};
\draw[tikzdefault,->] (phi).. controls +(140:3.5cm) and +(40:3.5cm) ..
(phi) node[pos=0.5,above] {$X^*$};
\end{tikzpicture}
}

\newcommand{\tzSphereProofIII}{
\begin{tikzpicture}[baseline=0cm]
\tikzsetup
\OrientedCircle[xscale=-1]
\node at (0,1.2) {$X$};
\end{tikzpicture}
}

\newcommand{\tzSphereProofIV}{
\begin{tikzpicture}[baseline=0cm]
\tikzsetup
\OrientedCircle[xscale=-1]
\node at (0,1.2) {$X^*$};
\end{tikzpicture}
}

\newcommand{\tzMoveMtwoI}{
\begin{tikzpicture}[baseline=0pt]
\tikzsetup
\node[morphism] (ph) at (0,0) {$\ph_\al$};
\node[morphism] (psi) at (2,0) {$\psi_\be$};
\node at (0,0.7) {$\dots$};
\node at (2,0.7) {$\dots$};
\draw[tikzdefault,->] (ph)-- +(120:1cm) node[pos=1.0,above,scale=0.8]
{$V_1$};
\draw[tikzdefault,->] (ph)-- +(60:1cm) node[pos=1.0,above,scale=0.8]
{$V_n$};
\draw[tikzdefault,<-] (psi)-- +(120:1cm) node[pos=1.0,above,scale=0.8]
{$W_m$};
\draw[tikzdefault,<-] (psi)-- +(60:1cm) node[pos=1.0,above,scale=0.8]
{$W_1$};
\draw[tikzdefault,->] (ph) -- (psi) node[pos=0.5,below,scale=0.8] {$X_i$}; 
\end{tikzpicture}
}

\newcommand{\tzMoveMtwoII}{
\begin{tikzpicture}[baseline=0pt]
\tikzsetup
\node[morphism] (ph) at (2,0) {$\ph^\al$};
\node[morphism] (psi) at (0,0) {$\psi^\be$};
\node at (0,0.7) {$\dots$};
\node at (2,0.7) {$\dots$};
\draw[tikzdefault,<-] (ph)-- +(120:1cm) node[pos=1.0,above,scale=0.8]
{$V_n$};
\draw[tikzdefault,<-] (ph)-- +(60:1cm) node[pos=1.0,above,scale=0.8]
{$V_1$};
\draw[tikzdefault,->] (psi)-- +(120:1cm) node[pos=1.0,above,scale=0.8]
{$W_1$};
\draw[tikzdefault,->] (psi)-- +(60:1cm) node[pos=1.0,above,scale=0.8]
{$W_m$};
\draw[tikzdefault,->] (psi) -- (ph) node[pos=0.5,below,scale=0.8] {$X^*_i$}; 
\end{tikzpicture}
}

\newcommand{\tzMoveMtwoIII}{
\begin{tikzpicture}[baseline=0pt]
\tikzsetup
\node[coupon] (ph) at (0,0) {$\ph_\al\ccc{X_i}\psi_\be$};
\draw[tikzdefault,->] (ph.150)-- +(110:1cm) node[pos=1.0,above left,scale=0.8]
{$V_1$};
\draw[tikzdefault,->] (ph.130)-- +(110:1cm) node[pos=1.0,above,scale=0.8]
{$V_n$};
\draw[tikzdefault,<-] (ph.50)-- +(70:1cm) node[pos=1.0,above,scale=0.8]
{$W_m$};
\draw[tikzdefault,<-] (ph.30)-- +(70:1cm) node[pos=1.0,above right,scale=0.8]
{$W_1$};
\end{tikzpicture}
}
\newcommand{\tzMoveMtwoIV}{
\begin{tikzpicture}[baseline=0pt]
\tikzsetup
\node[coupon] (ph) at (0,0) {$\psi^\be\ccc{X_i}\ph^\al$};
\draw[tikzdefault,<-] (ph.150)-- +(110:1cm) node[pos=1.0,above left,scale=0.8]
{$W_1$};
\draw[tikzdefault,<-] (ph.130)-- +(110:1cm) node[pos=1.0,above,scale=0.8]
{$W_m$};
\draw[tikzdefault,->] (ph.50)-- +(70:1cm) node[pos=1.0,above,scale=0.8]
{$V_m$};
\draw[tikzdefault,->] (ph.30)-- +(70:1cm) node[pos=1.0,above right,scale=0.8]
{$V_n$};
\end{tikzpicture}
}

\newcommand{\tzGluingI}{
\begin{tikzpicture}[baseline=0cm]
\tikzsetup
\path[use as bounding box] (-1,-1.5) rectangle (1,2.2);
\node[morphism] (psi) at (0,0) {$\psi_a$};
\node[morphism] (ph') at (0,1) {$\ph'_a$};
\node[morphism] (ph) at (0,-1) {$\ph_a$};
\draw[drinfeld center] (ph)
               .. controls +(135:1cm) and +(-135:1cm) ..
               (ph') node[pos=0.5, left, black] {$Z$};
\draw[tikzdefault,->] (psi)--(ph') node[pos=0.6,right=-2pt] {$A'$};
\draw[tikzdefault,->] (ph)--(psi) node[pos=0.4,right=-2pt] {$A$};
\draw[overline,->] (psi).. controls +(140:3.5cm) and +(40:3.5cm) ..
(psi) node[pos=0.5,above] {$k$};
\end{tikzpicture}
}

\newcommand{\tzGluingII}{
\begin{tikzpicture}[baseline=0cm]
\tikzsetup
\path[use as bounding box] (-1,-1.5) rectangle (1,2.2);
\node[morphism] (psi) at (0,0) {$\psi_b$};
\node[morphism] (ph') at (0,1) {$\ph'_b$};
\node[morphism] (ph) at (0,-1) {$\ph_b$};
\draw[drinfeld center] (ph)
               .. controls +(135:1cm) and +(-135:1cm) ..
               (ph') node[pos=0.5, left, black] {$Z^*$};
\draw[tikzdefault,->] (psi)--(ph') node[pos=0.6,right=-2pt] {$B'$};
\draw[tikzdefault,->] (ph)--(psi) node[pos=0.4,right=-2pt] {$B$};
\draw[overline,->] (psi).. controls +(140:3.5cm) and +(40:3.5cm) ..
(psi) node[pos=0.5,above] {$l$};
\end{tikzpicture}
}

\newcommand{\tzGluingIII}{
\begin{tikzpicture}[baseline=0cm]
\tikzsetup
\path[use as bounding box] (-1,-2.2) rectangle (5.5,3);
\node[morphism] (psi_a) at (0,0) {$\psi_a$};
\node[morphism] (C) at (1.5,0) {$\ph$};
\node[morphism] (psi_b) at (3,0) {$\psi_b$};
\node[morphism] (Cbar) at (4.5,0) {$\ph^*$};
\node[morphism] (ph_a') at (0,1) {$\ph'_a$};
\node[morphism] (ph_b') at (3,1) {$\ph'_b$};
\node[morphism] (ph_a) at (0,-1) {$\ph_a$};
\node[morphism] (ph_b) at (3,-1) {$\ph_b$};
\draw[drinfeld center] (ph_a)--(ph_b) node[pos=0.3,below, black] {$Z$};
\draw[drinfeld center] (ph_a')--(ph_b') node[pos=0.3,above, black] {$Z^*$};
\draw[tikzdefault,<-] (psi_a)-- (C) node[pos=0.6,above] {$k$} ;
\draw[tikzdefault,<-] (C)--(psi_b) node[pos=0.6,above] {$l$};
\draw[tikzdefault, <-] (psi_b)--(Cbar) node[pos=0.4,above] {$l$};
\draw[tikzdefault,->] (psi_a)--(ph_a') node[pos=0.5,right] {$A'$};
\draw[tikzdefault,->] (psi_b)--(ph_b') node[pos=0.5,left] {$B'$};
\draw[overline, <-] (C)--node[left] {$i$} +(0,0.5) arc (180:0:1.5cm) -- (Cbar);
\draw[tikzdefault,<-] (psi_a)--(ph_a) node[pos=0.4,right] {$A$};
\draw[tikzdefault,<-] (psi_b)--(ph_b) node[pos=0.4,left] {$B$};
\draw[overline, <-] (C)--node[left] {$j$} +(0,-0.5) arc (180:360:1.5cm) -- (Cbar);
\draw[tikzdefault,->] (psi_a).. controls +(130:4cm) and +(50:4cm) ..
(Cbar) node[pos=0.5,above] {$k$};
\end{tikzpicture}
}

\newcommand{\tzGluingIV}{
\begin{tikzpicture}[baseline=0cm]
\tikzsetup
\path[use as bounding box] (-1,-2.2) rectangle (5,2.7);

\node[morphism] (C_L) at (0,0) {$\psi_a$};
\coordinate (C) at (1.5,0);
\node[morphism] (C_R) at (3,0) {$\psi_b$};
\coordinate (Cbar) at (4.5,0);
\node[morphism] (psi_al') at (0,1) {$\ph'_a$};
\node[morphism] (psi_be') at (3,1) {$\ph'_b$};
\node[morphism] (psi_al) at (0,-1) {$\ph_a$};
\node[morphism] (psi_be) at (3,-1) {$\ph_b$};
\draw[drinfeld center] (psi_al)--(psi_be);
\draw[drinfeld center] (psi_al')--(psi_be');
\draw[tikzdefault,->] (C_L)--(psi_al') node[pos=0.6,right=-2pt] {$A'$};
\draw[tikzdefault,->] (C_R)--(psi_be') node[pos=0.6,right=-2pt] {$B'$};
\draw[overline,->] (C)-- +(0,1) arc (180:0:1.5cm) -- (Cbar)  node[right] {$j$};
\draw[tikzdefault,<-] (C_L)--(psi_al) node[pos=0.6,right=-2pt] {$A$};
\draw[tikzdefault,<-] (C_R)--(psi_be) node[pos=0.6,right=-2pt] {$B$};
\draw[overline] (C)-- +(0,-0.5) arc (180:360:1.5cm) -- (Cbar);
\draw[overline,->] (C_L).. controls +(140:3.5cm) and +(40:3.5cm) ..
(C_L)  node[pos=0.5,above] {$k$};
\draw[overline,->] (C_R).. controls +(140:3.5cm) and +(40:3.5cm) ..
(C_R) node[pos=0.5,above] {$l$};
\end{tikzpicture}
}

\newcommand{\tzGluingV}{
\begin{tikzpicture}
\tikzsetup

\begin{scope}[tikzdefault,xscale=0.7]
\coordinate (a1) at (60:0.5cm);\coordinate (a2) at ($(10cm,0)+(60:0.5cm)$);
\coordinate (b1) at ($(2.5cm,0)+(60:1.5cm)$);\coordinate (b2) at ($(7.5cm,0)+(60:1.5cm)$);
\coordinate (c1) at (0,-0.5);\coordinate (c2) at (10,-0.5);
\coordinate (d1) at (2.5,-1.5);\coordinate (d2) at (7.5,-1.5);

\path[cell] (114:0.5cm)--
             ($(2.5cm,0)+(114:1.5cm)$)
            arc (114:-114:1.5cm)
            --  (-114:0.5cm)
            arc (-114:114:0.5cm);
\path[cell] ($(10cm,0)+(66:0.5cm)$)--
             ($(7.5cm,0)+(66:1.5cm)$)
            arc (66:-66:1.5cm)
            --  ($(10cm,0)+(-66:0.5cm)$)
            arc (-66:66:0.5cm);

\node[dotnode]  at (a1) {};
\node[dotnode]  at (b1) {};
\node[dotnode]  at (a2) {};
\node[dotnode]  at (b2) {};
\OrientedCircle[xscale=0.5, yscale=-0.5] \node[left=-3pt]  at (0.5,0)  {$A$};
\node at (-2,0) {$D_a$};
\HalfDashedOrientedCircle[ xshift=2.5cm, xscale=1.5, yscale=-1.5] \node[right]  at (4,0)  {$A'$};
\draw[->,shorten > = 1pt] (a1)-- node[below] {$k$}(b1);
\draw[->,shorten > = 1pt] (a2)-- node[above] {$l$}(b2);
\HalfDashedOrientedCircle[xshift=10cm,scale=0.5] 
  \node[right]  at (10.5,0) {$B$};
  \node at (12,0) {$D_b$};
\HalfDashedOrientedCircle[xshift=7.5cm,scale=1.5] 
  \node[left]  at (9,0) {$B'$};

\end{scope}
\end{tikzpicture}
}

\newcommand{\tzGluingVI}{
\begin{tikzpicture}
\tikzsetup

\begin{scope}[tikzdefault,xscale=0.7]
\coordinate (a1) at (60:0.5cm);\coordinate (a2) at ($(10cm,0)+(60:0.5cm)$);
\coordinate (b1) at ($(2.5cm,0)+(60:1.5cm)$);\coordinate (b2) at ($(7.5cm,0)+(60:1.5cm)$);
\coordinate (c1) at (0,-0.5);\coordinate (c2) at (10,-0.5);
\coordinate (d1) at (2.5,-1.5);\coordinate (d2) at (7.5,-1.5);

\path[cell] (114:0.5cm)--
             ($(2.5cm,0)+(114:1.5cm)$)
            arc (114:-114:1.5cm)
            --  (-114:0.5cm)
            arc (-114:114:0.5cm);
\path[cell] ($(10cm,0)+(66:0.5cm)$)--
             ($(7.5cm,0)+(66:1.5cm)$)
            arc (66:-66:1.5cm)
            --  ($(10cm,0)+(-66:0.5cm)$)
            arc (-66:66:0.5cm);

\fill[yellow!20] (d1) arc (-90:90:1.5) -- (7.5,1.5) arc (90:-90:1.5)--cycle;

\node[dotnode]  at (a1) {};
\node[dotnode]  at (b1) {};
\node[dotnode]  at (a2) {};
\node[dotnode]  at (b2) {};
\OrientedCircle[xscale=0.5, yscale=-0.5] \node[left=-3pt]  at (0.5,0)  {$A$};
\HalfDashedOrientedCircle[ xshift=2.5cm, xscale=1.5, yscale=-1.5] \node[right]  at (4,0)  {$A'$};
\draw[->,shorten > = 1pt] (a1)--node [above] {$i$}(a2);
\draw[->,shorten > = 1pt] (b1)--node [above] {$j$}(b2);
\draw[thin,densely dashed] (c1)--(c2);
\draw[very thin] (d1)--(d2);
\draw[very thin] (2.5,1.5)--(7.5,1.5);
\draw[->,shorten > = 1pt] (a1)-- node[below] {$k$}(b1);
\draw[->,shorten > = 1pt] (a2)-- node[above] {$l$}(b2);
\HalfDashedOrientedCircle[xshift=10cm,scale=0.5] 
  \node[right]  at (10.5,0) {$B$};
\HalfDashedOrientedCircle[xshift=7.5cm,scale=1.5] 
  \node[left]  at (9,0) {$B'$};

\end{scope}
\end{tikzpicture}
}

\newcommand{\tzCylinderI}{
\begin{tikzpicture}
\tikzsetup

\begin{scope}[tikzdefault,xscale=0.7]
\coordinate (a1) at (0,0.5);\coordinate (a2) at (7,0.5);
\coordinate (b1) at (0,1.5);\coordinate (b2) at (7,1.5);
\coordinate (c1) at (0,-0.5);\coordinate (c2) at (7,-0.5);
\coordinate (d1) at (0,-1.5);\coordinate (d2) at (7,-1.5);

\path[cell] (a1) arc(-270:89.9:0.5) -- ++(0,1) arc
(89.9:-270:1.5)--cycle;

\fill[yellow!20] (d1) arc (-90:90:1.5) -- (b2) arc (90:-90:1.5)--cycle;

\node[dotnode]  at (a1) {};
\node[dotnode]  at (b1) {};
\node[dotnode]  at (a2) {};
\node[dotnode]  at (b2) {};
\OrientedCircle[scale=0.5] \node[left=-3pt]  at (0.5,0)  {$A$};
\OrientedCircle[scale=1.5] \node[right]  at (1.5,0)  {$A'$};
\draw[->,shorten > = 1pt] (a1)--node [above] {$i$}(a2);
\draw[->,shorten > = 1pt] (b1)--node [above] {$j$}(b2);
\draw[thin,densely dashed] (c1)--(c2);
\draw[thin] (d1)--(d2);
\draw[->,shorten > = 1pt] (a1)-- node[right] {$k$}(b1);
\draw[->,shorten > = 1pt] (a2)-- node[right] {$l$}(b2);
\OrientedCircle[xshift=7cm,scale=0.5,densely dashed] 
  \node[left=-3pt]  at (7.5,0) {$B$};
\HalfDashedOrientedCircle[xshift=7cm,scale=1.5] 
  \node[right]  at (8.5,0) {$B'$};

\end{scope}
\end{tikzpicture}
}

\newcommand{\tzCL}{
\begin{tikzpicture}[baseline]
\tikzsetup
\path[cell, even odd rule] (0,0) circle (0.5) (0,0) circle (1.5); 
\OrientedCircle[scale=0.5] 
        \node[left=-3pt] at (0.5,0) {$A$};
\OrientedCircle[scale=1.5] 
        \node[right] at (1.5,0) {$A'$};
\node[dotnode]  (a1) at (0,0.5){};
\node[dotnode]  (a2) at (0,1.5){};
\draw[->,tikzdefault] (a1) -- node[left] {$k$} (a2);
\end{tikzpicture}
}

\newcommand{\tzCR}{
\begin{tikzpicture}[baseline]
\tikzsetup
\path[cell, even odd rule] (0,0) circle (0.5) (0,0) circle (1.5); 
\OrientedCircle[xscale=-0.5, yscale=0.5] 
        \node[right=-3pt] at (-0.5,0) {$B$};
\OrientedCircle[xscale=-1.5,yscale=1.5] 
        \node[left] at (-1.5,0) {$B'$};
\node[dotnode]  (b1) at (0,0.5){};
\node[dotnode]  (b2) at (0,1.5){};
\draw[->,tikzdefault] (a1) -- node[left] {$l$} (a2);
\end{tikzpicture}
}

\newcommand{\tzCin}{
\begin{tikzpicture}[baseline=0.5cm]
\tikzsetup
\coordinate (a) at (0,0);\coordinate (b) at (0,1);
\coordinate (c) at (1.5,1);\coordinate (d) at (1.5,0);
\path[cell] (a)--(b)--(c)--(d)--cycle;
\draw[tikzdefault,<-] (a)-- node [left] {$A$} (b);
\draw[tikzdefault,->] (b)-- node [above] {$i$} (c);
\draw[tikzdefault,<-] (d)-- node [right] {$B$} (c);
\draw[tikzdefault,->] (a)-- node [below] {$i$} (d);
\end{tikzpicture}
}
\newcommand{\tzCout}{
\begin{tikzpicture}[baseline=0.5cm]
\tikzsetup
\coordinate (a) at (0,0);\coordinate (b) at (0,1);
\coordinate (c) at (1.5,1);\coordinate (d) at (1.5,0);
\path[cell] (a)--(b)--(c)--(d)--cycle;
\draw[tikzdefault,->] (a)-- node [left] {$A'$} (b);
\draw[tikzdefault,->] (b)-- node [above] {$j$} (c);
\draw[tikzdefault,->] (d)-- node [right] {$B'$} (c);
\draw[tikzdefault,->] (a)-- node [below] {$j$} (d);
\end{tikzpicture}
}
\newcommand{\tzC}{
\begin{tikzpicture}[baseline=0.5cm]
\tikzsetup
\coordinate (a) at (0,0);\coordinate (b) at (0,1);
\coordinate (c) at (1.5,1);\coordinate (d) at (1.5,0);
\path[cell] (a)--(b)--(c)--(d)--cycle;
\draw[tikzdefault,->] (a)-- node [left] {$k$} (b);
\draw[tikzdefault,->] (b)-- node [above] {$j$} (c);
\draw[tikzdefault,->] (d)-- node [right] {$l$} (c);
\draw[tikzdefault,->] (a)-- node [below] {$i$} (d);
\end{tikzpicture}}
\newcommand{\tzCube}{
\begin{tikzpicture}[baseline=0.5cm]
\tikzsetup
\coordinate (a1) at (0,0);\coordinate (b1) at (1.2,1);
\coordinate (d1) at (2,0);\coordinate (c1) at (3.2,1);
\coordinate (a2) at (0,2);\coordinate (b2) at (1.2,3);
\coordinate (d2) at (2,2);\coordinate (c2) at (3.2,3);
\draw[tikzdefault,->] (b1)-- (c1) node [pos=0.8,above] {$i$} ;
\draw[tikzdefault,->] (b2)-- node [above] {$j$} (c2);
\draw[tikzdefault,->] (b1)-- (b2) node [pos=0.3,left] {$k$};
\draw[tikzdefault,->] (c1)-- node [right] {$l$} (c2);
\draw[overline] (a2)-- (d2) node [pos=0.8, above] {$j$} ;
\draw[overline,->] (d1)-- (d2) node [pos=0.7,right] {$l$};
\draw[tikzdefault,->] (a2)-- (d2); 
\draw[tikzdefault,->] (a1)-- (a2)node [pos=0.3,left] {$k$};
\draw[tikzdefault,->] (a1)--  (d1) node [pos=0.6,above] {$i$};
\draw[tikzdefault,->] (a2)--  (b2) node [pos=0.7,left]{$A'$};
\draw[tikzdefault,->] (a1)--  (b1) node [pos=0.7,left] {$A$};
\draw[tikzdefault,->] (d1)-- (c1) node [pos=0.3,right] {$B$};
\draw[tikzdefault,->] (d2)-- (c2) node[pos=0.3,right] {$B'$};
\end{tikzpicture}
}

\newcommand{\tzDualGraphCylII}{
\begin{tikzpicture}[baseline=0.5cm]
\tikzsetup
\node[morphism] (C_L) at (0,0) {$\psi_a$};
\node[morphism] (C) at (2,0) {$\ph$};
\node[morphism] (C_R) at (4,0) {$\psi_b$};
\node[morphism] (Cbar) at (6,0) {$\ph'$};
\node[morphism] (C_out) at (3,3) {$\psi_{out}$};
\node[morphism] (C_in) at (3,-3) {$\psi_{in}$};
\draw[tikzdefault,<-] (C_L)-- (C) node[pos=0.6,above] {$k$} ;
\draw[tikzdefault,<-] (C)--(C_R) node[pos=0.6,above] {$l$};
\draw[tikzdefault, <-] (C_R)--(Cbar) node[pos=0.4,above] {$l$};
\draw[tikzdefault,<-] (C_L)--(C_out) node[pos=0.4,left=2pt] {$A'$};
\draw[tikzdefault,<-] (C)--(C_out) node[pos=0.4,left] {$j$};
\draw[tikzdefault,->] (C_R)--(C_out) node[pos=0.4,right] {$B'$};
\draw[tikzdefault,->] (Cbar)--(C_out) node[pos=0.4,right=3pt] {$j$};
\draw[tikzdefault,<-] (C_in)--(C_L) node[pos=0.6,left=2pt] {$A$};
\draw[tikzdefault,->] (C_in)--(C) node[pos=0.6,left] {$i$};
\draw[tikzdefault,->] (C_in)--(C_R) node[pos=0.6,right] {$B$};
\draw[tikzdefault,<-] (C_in)--(Cbar) node[pos=0.6,right=3pt] {$i$};
\draw[tikzdefault,->] (C_L).. controls +(90:5cm) and +(90:5cm) .. (Cbar);
\node at (3,4.2) {$k$};
\end{tikzpicture}
}
\newcommand{\tzNsphereI}{
\begin{tikzpicture}[baseline=0cm]
\tikzsetup
\draw (0,-1.5) rectangle (7,-2.2); \coordinate (bot) at (0,-1.5);
\draw (0,1.5) rectangle (7,2.2);\coordinate (top) at (0,1.5); 
\node at (3.5, -1.8) {$\ph_0$};
\node at (3.5, 1.8) {$\ph'_0$};
\node[morphism,scale=0.8] (X_1) at (0.3,0) {$\al_1$};
\node[morphism, scale=0.8] (V_1) at (1,0) {$\beta_1$};
\node[morphism, scale=0.8] (X_1*) at (1.7,0){$\al_1^*$};
\node[morphism, scale=0.8] (X_n) at (5.3,0) {$\al_n$};
\node[morphism, scale=0.8] (V_n) at (6,0) {$\beta_n$};
\node[morphism, scale=0.8] (X_n*) at (6.7,0) {$\al_n^*$};
\draw[->,tikzdefault] (X_1|-bot) -- (X_1) 
                  node[pos=0.5,left=-1pt,scale=0.7]{$i_1$};
 \draw[->,tikzdefault] (X_1)--(X_1|-top)
                  node[pos=0.5,left=-1pt,scale=0.7]{$j_1$};;
\draw[->,tikzdefault] (V_1|-bot) -- (V_1) 
                  node[pos=0.5,left=-1pt,scale=0.7]{$U_1$};
 \draw[->,tikzdefault] (V_1) -- (V_1|-top) 
                  node[pos=0.5,left=-1pt,scale=0.7]{$V_1$};
\draw[<-,tikzdefault] (X_1*|-bot) -- (X_1*) 
                  node[pos=0.5,right=-1pt,scale=0.7]{$i_1$};
 \draw[<-,tikzdefault] (X_1*) --(X_1*|-top) 
                  node[pos=0.5,right=-1pt,scale=0.7]{$j_1$};
\draw[->,tikzdefault] (X_n|-bot) -- (X_n) 
                  node[pos=0.5,left=-1pt,scale=0.7]{$i_n$};
\draw[->,tikzdefault] (X_n)--(X_n|-top) 
                  node[pos=0.5,left=-1pt,scale=0.7]{$j_n$};
\draw[->,tikzdefault] (V_n|-bot) -- (V_n) 
                  node[pos=0.5,left=-1pt,scale=0.7]{$U_n$};
\draw[->,tikzdefault] (V_n) --(V_n|-top) 
                  node[pos=0.5,left=-1pt,scale=0.7]{$V_n$};
\draw[<-,tikzdefault] (X_n*|-bot) -- (X_n*) 
                  node[pos=0.5,right=-1pt,scale=0.7]{$i_n$};
\draw[<-,tikzdefault] (X_n*) --(X_n*|-top) 
                  node[pos=0.5,right=-1pt,scale=0.7]{$j_n$};
\draw[->,tikzdefault] (X_1)-- (V_1);
\draw[->,tikzdefault] (V_1)-- (X_1*);
\draw[->,tikzdefault] (X_1*)-- +(0:0.7cm)
              node[pos=0.8,above,scale=0.7]{$l$};
\draw[<-,tikzdefault] (X_n)-- +(180:0.7cm)
              node[pos=0.8,above,scale=0.7]{$l$};
\draw[->,tikzdefault] (X_n)-- (V_n);
\draw[->,tikzdefault] (V_n)-- (X_n*);
\draw[<-,tikzdefault] (X_1.west) arc (90:270:1.3cm)
               node[pos=0.5,above left=1pt , scale=0.7] {$l$} 
               coordinate (leftbot);
\draw[->,tikzdefault] (X_n*.east) arc (90:-90:1.3cm)  coordinate
(rightbot);
\draw[<-,tikzdefault] (leftbot)--(rightbot);
\end{tikzpicture}
}
\newcommand{\tzNsphereII}{
\begin{tikzpicture}[baseline=3cm]
\tikzsetup
\node[morphism] (Y_1bot) at (0,1.5) {$\ph_i$};
\node[morphism] (Y_1side) at (0,3) {$\be_i^*$};
\node[morphism] (Y_1top) at (0,4.5) {$\ph'_i$};
\draw[tikzdefault,->] (Y_1bot)--(Y_1side) 
               node[midway, left,scale=0.8] {$U_i$};
\draw[tikzdefault,->] (Y_1side)--(Y_1top) 
               node[midway, left,scale=0.8] {$V_i$};
\draw[drinfeld center] (Y_1bot) .. 
      controls +(45:1cm) and +(-45:1cm) 
      ..(Y_1top) node[pos=0.3, right,scale=0.8] {$Y_i$};
\draw[tikzdefault,overline] 
 (Y_1side.east) 
    .. controls +(0:1.5cm) and + (0:1.5cm) ..
 (0,5.5) node [pos=0.5,right, scale=0.8] {$k_i$}
.. controls +(180:1.5cm) and + (180:1.5cm).. 
 (Y_1side.west);
\end{tikzpicture}
}
\newcommand{\tzNsphereIII}{
\begin{tikzpicture}[baseline=0cm]
\tikzsetup
\draw (0,-1.5) rectangle (7,-2.2); \coordinate (bot) at (0,-1.5);
\draw (0,1.5) rectangle (7,2.2);\coordinate (top) at (0,1.5); 
\node at (3.5, -1.8) {$\ph_0$};
\node at (3.5, 1.8) {$\ph'_0$};
\node[morphism,scale=0.8] (X_1) at (0.3,0) {$\al_1$};
\node[morphism, scale=0.8] (V_1bot) at (1,-0.7) {$\ph_1$};
\node[morphism, scale=0.8] (V_1top) at (1,0.7) {$\ph'_1$};
\node[morphism, scale=0.8] (X_1*) at (1.7,0){$\al_1^*$};
\node[morphism, scale=0.8] (X_n) at (5.3,0) {$\al_n$};
\node[morphism, scale=0.8] (V_nbot) at (6,-0.7) {$\ph_n$};
\node[morphism, scale=0.8] (V_ntop) at (6,0.7) {$\ph'_n$};
\node[morphism, scale=0.8] (X_n*) at (6.7,0) {$\al_n^*$};
\draw[->,tikzdefault] (X_1|-bot) -- (X_1) 
                  node[pos=0.5,left=-1pt,scale=0.7]{$i_1$};
 \draw[->,tikzdefault] (X_1)--(X_1|-top)
                  node[pos=0.5,left=-1pt,scale=0.7]{$j_1$};;
\draw[->,tikzdefault] (V_1bot|-bot) -- (V_1bot) 
                  node[pos=0.5,left=-1pt,scale=0.7]{$U_1$};
\draw[drinfeld center] (V_1bot) -- (V_1top) 
                  node[pos=0.5,left=-1pt,scale=0.7]{$Y_1$};
 \draw[->,tikzdefault] (V_1top) -- (V_1top|-top) 
                  node[pos=0.5,left=-1pt,scale=0.7]{$V_1$};
\draw[<-,tikzdefault] (X_1*|-bot) -- (X_1*) 
                  node[pos=0.5,right=-1pt,scale=0.7]{$i_1$};
 \draw[<-,tikzdefault] (X_1*) --(X_1*|-top) 
                  node[pos=0.5,right=-1pt,scale=0.7]{$j_1$};
\draw[->,tikzdefault] (X_n|-bot) -- (X_n) 
                  node[pos=0.5,left=-1pt,scale=0.7]{$i_n$};
\draw[<-,tikzdefault] (X_n)--(X_n|-top) 
                  node[pos=0.5,left=-1pt,scale=0.7]{$j_n$};
\draw[->,tikzdefault] (V_nbot|-bot) -- (V_nbot) 
                  node[pos=0.5,left=-1pt,scale=0.7]{$U_n$};
\draw[->,tikzdefault] (V_ntop) --(V_ntop|-top) 
                  node[pos=0.5,left=-1pt,scale=0.7]{$V_n$};
\draw[drinfeld center] (V_nbot) -- (V_ntop) 
                  node[pos=0.5,left=-1pt,scale=0.7]{$Y_1$};
\draw[<-,tikzdefault] (X_n*|-bot) -- (X_n*) 
                  node[pos=0.5,right=-1pt,scale=0.7]{$i_n$};
\draw[->,tikzdefault] (X_n*) --(X_n*|-top) 
                  node[pos=0.5,right=-1pt,scale=0.7]{$j_n$};
\draw[->,overline] (X_1)-- (X_1*) node[pos=0.3, above,scale=0.8] {$k_1$};
\draw[->,tikzdefault] (X_1*)-- +(0:0.7cm)
              node[pos=0.8,above,scale=0.7]{$l$};
\draw[<-,tikzdefault] (X_n)-- +(180:0.7cm)
              node[pos=0.8,above,scale=0.7]{$l$};
\draw[->,overline] (X_n)-- (X_n*) node[pos=0.3, above,scale=0.8] {$k_n$};
\draw[<-,tikzdefault] (X_1.west) arc (90:270:1.3cm)
               node[pos=0.5,above left=1pt , scale=0.7] {$l$} 
               coordinate (leftbot);
\draw[->,tikzdefault] (X_n*.east) arc (90:-90:1.3cm)  coordinate
(rightbot);
\draw[<-,tikzdefault] (leftbot)--(rightbot);
\end{tikzpicture}
}

\newcommand{\tzNsphereIV}{
\begin{tikzpicture}[baseline=0cm]
\tikzsetup
\coordinate (bot) at (0,-1.5);
\coordinate (top) at (0,1.5); 
\node at (3.5, -1.8) {$\ph$};
\node at (3.5, 1.8) {$\ph'$};
\coordinate (X_1) at (0.3,0);
\coordinate (V_1) at (1,0);
\coordinate (X_1*) at (1.7,0);
\coordinate (X_n) at (5.3,0);
\coordinate (V_n) at (6,0);
\coordinate (X_n*) at (6.7,0);
\draw[drinfeld center] (V_1|-bot) -- (V_1|-top) 
                  node[pos=0.1,left=-1pt,scale=0.7,black]{$Y_1$};
\draw[drinfeld center] (V_n|-bot) -- (V_n|-top) 
                  node[pos=0.1,left=-1pt,scale=0.7,black]{$Y_n$};
\draw[->,overline] (X_1|-bot) -- node[pos=0.9,left,scale=0.8]{$i_1$} 
                  +(0,0.3cm) arc (180:0:0.7cm) -- (X_1*|-bot);
\draw[<-,overline] (X_1|-top) -- node[pos=0.9,left,scale=0.8]{$j_1$}
                  +(0,-0.3cm) arc (180:360:0.7cm) -- (X_1*|-top);
\draw[->,overline] (X_n|-bot) -- node[pos=0.9,left,scale=0.8]{$i_n$} 
                  +(0,0.3cm) arc (180:0:0.7cm) -- (X_n*|-bot);
\draw[<-,overline] (X_n|-top) -- node[pos=0.9,left,scale=0.8]{$j_n$}
                  +(0,-0.3cm) arc (180:360:0.7cm) -- (X_n*|-top);
\draw[overline] (X_1)--(X_n*) node[pos=0.5, above,scale=0.8] {$l$};
\draw[<-,tikzdefault] (X_1) arc (90:270:1.3cm)
               coordinate (rightbot)
               node[pos=0.5,above left=1pt , scale=0.7] {$l$};
\draw[->,tikzdefault] (X_n*) arc (90:-90:1.3cm) coordinate (rightbot);
\draw[tikzdefault] (leftbot)--(rightbot);
\draw (0,-1.5) rectangle (7,-2.2);
\draw (0,1.5) rectangle (7,2.2);
\node at (3.5, -1) {$\dots$};
\node at (3.5, 1) {$\dots$};
\end{tikzpicture}
}
\newcommand{\tzNsphereV}{
\begin{tikzpicture}[baseline=0cm]
\tikzsetup
\coordinate (bot) at (0,-1.5);
\coordinate (top) at (0,1.5); 
\node at (3.5, -1.8) {$\ph$};
\coordinate (X_1) at (0.3,0);
\coordinate (V_1) at (1,0);
\coordinate (X_1*) at (1.7,0);
\coordinate (X_n) at (5.3,0);
\coordinate (V_n) at (6,0);
\coordinate (X_n*) at (6.7,0);
\draw[drinfeld center] (V_1|-bot) -- (V_1|-top) 
                  node[pos=0.1,left=-1pt,scale=0.7,black]{$Y_1$};
\draw[drinfeld center] (V_n|-bot) -- (V_n|-top) 
                  node[pos=0.1,left=-1pt,scale=0.7,black]{$Y_n$};
\draw[->,overline] (X_1|-bot) -- node[pos=0.9,left,scale=0.8]{$i_1$} 
                  +(0,0.3cm) arc (180:0:0.7cm) -- (X_1*|-bot);
\draw[<-,overline] (X_1|-top) -- node[pos=0.9,left,scale=0.8]{$j_1$}
                  +(0,-0.3cm) arc (180:360:0.7cm) -- (X_1*|-top);
\draw[->,overline] (X_n|-bot) -- node[pos=0.9,left,scale=0.8]{$i_n$} 
                  +(0,0.3cm) arc (180:0:0.7cm) -- (X_n*|-bot);
\draw[<-,overline] (X_n|-top) -- node[pos=0.9,left,scale=0.8]{$j_n$}
                  +(0,-0.3cm) arc (180:360:0.7cm) -- (X_n*|-top);
\draw[overline] (X_1)--(X_n*) node[pos=0.5, above,scale=0.8] {$l$};
\draw[<-,tikzdefault] (X_1) arc (90:270:1.3cm)
               coordinate (leftbot)
               node[pos=0.5,above left=1pt , scale=0.7] {$l$};
\draw[->,tikzdefault] (X_n*) arc (90:-90:1.3cm) coordinate (rightbot);
\draw[tikzdefault] (leftbot)--(rightbot);
\draw (0,-1.5) rectangle (7,-2.2);
\node at (3.5, -1) {$\dots$};
\end{tikzpicture}
}
\newcommand{\tzNsphereVI}{
\begin{tikzpicture}[baseline=0cm]
\tikzsetup
\coordinate (bot) at (0,-1);
\coordinate (top) at (0,1); 
\node at (3.5, -1.3) {$\psi$};
\coordinate (X_1) at (0.3,0);
\coordinate (V_1) at (1,0);
\coordinate (X_1*) at (1.7,0);
\coordinate (X_n) at (5.3,0);
\coordinate (V_n) at (6,0);
\coordinate (X_n*) at (6.7,0);
\draw[drinfeld center] (V_1|-bot) -- (V_1|-top) 
                  node[pos=0.1,left=-1pt,scale=0.7,black]{$Y_1$};
\draw[drinfeld center] (V_n|-bot) -- (V_n|-top) 
                  node[pos=0.1,left=-1pt,scale=0.7,black]{$Y_n$};
\draw[<-,overline] (X_1|-top) -- node[pos=0.9,left,scale=0.8]{$j_1$}
                  +(0,-0.3cm) arc (180:360:0.7cm) -- (X_1*|-top);
\draw[<-,overline] (X_n|-top) -- node[pos=0.9,left,scale=0.8]{$j_n$}
                  +(0,-0.3cm) arc (180:360:0.7cm) -- (X_n*|-top);
\draw (0,-1) rectangle (7,-1.7);
\node at (3.5, -0.5) {$\dots$};
\end{tikzpicture}
}

\newtheorem*{theorem*}{Theorem}
\newtheorem{theorem}{Theorem}[section]
\newtheorem{lemma}[theorem]{Lemma}

\newtheorem{corollary}[theorem]{Corollary}

\theoremstyle{definition}
\newtheorem{definition}[theorem]{Definition}

\newtheorem{example}[theorem]{Example}

\theoremstyle{remark}

\numberwithin{equation}{section}
\newcommand{\firef}[1]{Figure~{\rm\ref{#1}}}
\newcommand{\thref}[1]{Theorem~{\rm\ref{#1}}}

\newcommand{\leref}[1]{Lemma~{\rm\ref{#1}}}
\newcommand{\coref}[1]{Corollary~{\rm\ref{#1}}}

\newcommand{\deref}[1]{Definition~{\rm\ref{#1}}}
\newcommand{\exref}[1]{Example~{\rm\ref{#1}}}

\newcommand{\seref}[1]{Section~{\rm\ref{#1}}}

\newcommand{\fig}[1]
{\raisebox{-0.5\height}%
{\includegraphics{#1}}}
\newcommand{\figscale}[2]
{\raisebox{-0.5\height}%
{\includegraphics[scale=#1]{#2}}}
\newcommand{\cc}[1]{\underset{\scriptstyle #1}{\circ}}
\newcommand{\ccc}[1]{\underset{\scriptstyle #1}{\bullet}}

\newcommand{\ov}{\overline}
\newcommand{\del}{\partial}
\newcommand{\<}{\langle}
\renewcommand{\>}{\rangle}
\newcommand{\isoto}{\xrightarrow{\sim}}       
\newcommand{\xxto}{\xrightarrow}              

\newcommand{\one}{\mathbf{1}}
\newcommand{\CC}{\mathbb{C}}       
\newcommand{\kk}{\mathbf{k}}       
\newcommand{\Z}{\mathbb{Z}}       
\newcommand{\R}{\mathbb{R}}       
\newcommand{\N}{\mathcal{N}}       
\newcommand{\DD}{\mathcal{D}}      
\newcommand{\C}{\mathcal{C}}      
\newcommand{\M}{\mathcal{M}}      
\newcommand{\Vect}{\mathcal{V}ec}  
\newcommand{\VecG}{\mathcal{V}ec^G}  

\newcommand{\al}{\alpha}
\newcommand{\be}{\beta}
\newcommand{\ga}{\gamma}
\newcommand{\Ga}{\Gamma}

\newcommand{\ph}{\varphi}
\newcommand{\Ph}{\Phi} 

\newcommand{\Si}{\Sigma}

\newcommand{\g}{\mathfrak{g}}                  

\DeclareMathOperator{\Rep}{Rep}

\DeclareMathOperator{\Int}{Int}
\DeclareMathOperator{\Irr}{Irr}
\DeclareMathOperator{\id}{id}

\DeclareMathOperator{\Hom}{Hom}

\DeclareMathOperator{\tr}{tr}

\DeclareMathOperator{\ev}{ev} 
\DeclareMathOperator{\coev}{coev} 
\DeclareMathOperator{\Obj}{Obj}

\begin{document}

\title{Turaev-Viro invariants as an extended TQFT}

\author{Benjamin Balsam}
\author{Alexander Kirillov, Jr.}
   \address{Department of Mathematics, SUNY at Stony Brook, 
            Stony Brook, NY 11794, USA}
    \email{balsam@math.sunysb.edu}
    \email{kirillov@math.sunysb.edu}
    \urladdr{http://www.math.sunysb.edu/\textasciitilde kirillov/}
\thanks{This  work was partially suported by NSF grant DMS-0700589 }

\begin{abstract}
In this paper we show how one can extend Turaev-Viro invariants,
defined for an arbitrary spherical fusion category $\C$, to 3-manifolds
with corners. We demonstrate that this gives an extended TQFT which
conjecturally coincides with the Reshetikhin-Turaev TQFT corresponding to
the Drinfeld center $Z(\C)$. In the present paper we give a partial proof
of this statement. 

\end{abstract}

\maketitle
\section*{Introduction}
Turaev--Viro (TV) invariants of 3-manifolds $Z_{TV}(M)$ were defined by
Turaev and Viro in \ocite{TV} using a quantum analog of 6j symbols for
$sl_2$. In the same paper it was shown that these invariants can be
extended to a 3-dimensional TQFT. 

Later, Barrett and Westbury \ocite{barrett} showed that these
invariants can be defined for any monoidal category $\C$ possessing a
suitable notion of duality (``spherical category''). In particular, they
can be defined for the category of $G$-graded vector spaces, where $G$ is a
finite group. In this special case, the resulting TQFT coincides with the
version of Chern--Simons theory with the finite gauge group $G$, described
in \ocite{freed-quinn} (or in more modern language, in \ocite{FHLT}); in
physics literature, this theory is also known as the Levin--Wen model.

In the case when the category $\C$ is not only monoidal but in fact
modular (in particular, braided), there is another 3-dimensional TQFT
based on $\C$, namely Reshetikhin--Turaev TQFT. It was shown in
\ocite{turaev} that in this case, one has
$$
Z_{TV, \C}(M)=Z_{RT,\C}(M)Z_{RT,\C}(\ov{M})
$$
where $\ov{M}$ is $M$ with opposite orientation. In particular, if $\C$ is
unitary category over $\CC$, then $Z_{TV, \C}(M)=|Z_{RT,\C}(M)|^2$.

It has been conjectured that in the general case, when $\C$ is a spherical
(but not necessarily modular) category, one has
$$
Z_{TV, \C}(M)=Z_{RT, Z(\C)}(M)
$$
where $Z(\C)$ is the so-called Drinfeld center of $\C$ (see
\seref{s:center}); moreover, this extends to an isomorphism of the
corresponding TQFTs. Some partial results in this direction can be found,
for example, in \ocite{muger2}; however, the full statement remained a
conjecture.

The current paper is the first in a series giving a proof of this
conjecture for an arbitrary spherical category $\C$ over an algebraically
closed field $\kk$ of characteristic zero. In the current paper, we extend
TV invariants to 3-manifolds with corners of codimension 2 (or, which is
closely related, to 3-manifolds with framed tangles inside); in the
language of \ocite{lurie}, we construct a 3-2-1  extension of TV theory.
This extension satisfies $Z_{TV}(S^1)=Z(\C)$: boundary circles of
2-surfaces
should be colored by objects of $Z(\C)$. We also show that for an
$n$-punctured sphere, the
resulting vector space coming from this extended TV theory coincides with
the one coming from Reshetikhin-Turaev theory based on $Z(\C)$.

Thsi extended theory is related to the one
suggested in \ocite{turaev94}; however, that paper only considered the case
when $\C$ itself is modular and the components of links are labeled by
objects of $\C$, not $\Z(\C)$. We will investigate the relation between
these constructions in forthcoming papers. 

After the preliminary version of this paper was posted, a proof of formula
$Z_{TV, \C}(M)=Z_{RT, Z(\C)}(M)$ was announced by Turaev
\ocite{turaev-conference}; howevere, details of this proof are not yet
available.

\subsection*{Acknowledgments}
The authors would like  to thank
Oleg Viro, Victor Ostrik, Kevin Costello and Owen Gwilliam
for helpful suggestions and discussions.

\section{Preliminaries I: spherical categories}\label{s:prelim}

In this section we collect notation and some facts about spherical
categories.

We fix   an algebraically  closed field  $\kk$ of characteristic 0 and
 denote by $\Vect$ the category of finite-dimensional vector spaces over
$\kk$.

Throughout the paper, $\C$ will denote a spherical fusion category over 
 $\kk$. We refer the reader
to the paper \ocite{drinfeld} for the definitions and properties of such
categories. Note that we are not requiring a braiding on $\C$. 

In particular, $\C$ is semisimple with finitely many
isomorphism classes of simple objects. We will denote by $\Irr(\C)$ the
set of isomorphism classes of simple objects. We will also denote by $\one$
the unit object in $\C$ (which is simple).

Two main examples of spherical categories are the category $\VecG$ of
finite-dimensional $G$-graded vector spaces (where $G$ is a finite group)
and the category $\Rep(U_q\g)$ which is the semisimple part  of the
category of representations of a quantum group $U_q\g$ at a root of unity;
this last category is actually modular, but we will not be using this.

To simplify the notation, we will assume that $\C$ is a strict pivotal
category, i.e. that $V^{**}=V$. As is well-known, this is not really a
restriction, since any pivotal category is equivalent to a strict pivotal
category.

We will denote, for an object $X$ of $\C$, by 
$$
d_X=\dim X\in \kk
$$
its categorical dimension; it is known that for simple $X$, $d_X$ is
non-zero. We will fix, for any simple object $X_i\in \C$, a choice of
square root $\sqrt{d_X}$ so that for $X=\one$, $\sqrt{d_\one}=1$ and that
for any simple $X$, $\sqrt{d_X}=\sqrt{d_{X^*}}$. 

We will also denote
\begin{equation}\label{e:DD}
\DD=\sqrt{\sum_{x\in \Irr(\C)}d_X^2}
\end{equation}
(throughout the paper, we fix a choice of the square root). Note that by
results of \ocite{ENO}, $\DD\ne 0$.

We define the
functor $\C^{\boxtimes n}\to \Vect$ by 
\begin{equation}\label{e:vev}
 \<V_1,\dots,V_n\>=\Hom_\C(\one, V_1\otimes\dots\otimes V_n)
\end{equation}
for any collection $V_1,\dots, V_n$ of objects of $\C$. Note that pivotal
structure gives functorial isomorphisms 
\begin{equation}\label{e:cyclic}
 z\colon\<V_1,\dots,V_n\>\simeq \<V_n, V_1,\dots,V_{n-1}\>
\end{equation}
such that $z^n=\id$ (see \ocite{BK}*{Section 5.3}); thus, up to a canonical
isomorphism,
the space $\<V_1,\dots,V_n\>$ only depends on the cyclic order of
$V_1,\dots, V_n$. 

We have a natural composition map 
\begin{equation}\label{e:composition}
\begin{aligned}
 \<V_1,\dots,V_n, X\>\otimes\<X^*, W_1,\dots,
W_m\>&\to\<V_1,\dots,V_n, W_1,\dots, W_m\>\\
\ph\otimes\psi\mapsto \ph\cc{X}\psi= \ev_X\circ (\ph\otimes\psi)
\end{aligned}
\end{equation}
where $\ev_X\colon X\otimes  X^*\to \one$ is the evaluation morphism. It
follows from semisimplicity of $\C$ that direct sum of these composition
maps gives a functorial isomorphism
\begin{equation}\label{e:gluing1}
 \bigoplus_{X\in \Irr(\C)} \<V_1,\dots,V_n, X\>\otimes\<X^*, W_1,\dots,
W_m\>\simeq\<V_1,\dots,V_n, W_1,\dots, W_m\>.
\end{equation}

Note that for any objects $A,B\in \Obj \C$, we have a 
non-degenerate pairing $\Hom_\C(A,B)\otimes \Hom_\C(A^*,B^*)\to \kk$ 
defined by 
\begin{equation}\label{e:pairing}
(\ph, \ph')=(\one\xxto{\coev_A}A\otimes A^*\xxto{\ph\otimes \ph'}
  B\otimes B^*\xxto{\ev_B}\one)
\end{equation}
In particular, this gives us a non-degenerate pairing 
$\<V_1,\dots,V_n\>\otimes \<V_n^*,\dots,V_1^*\>\to \kk$ and thus,
 functorial isomorphisms 
\begin{equation}\label{e:dual}
\<V_1,\dots,V_n\>^*\simeq \<V_n^*,\dots,V_1^*\>
\end{equation}
compatible with the cyclic permutations \eqref{e:cyclic}.

We will frequently use graphical representations of morphisms in the
category $\C$, using tangle diagrams as in \ocite{turaev} or \ocite{BK}.
However, our convention is that of \ocite{BK}: a tangle with $k$ strands
labeled $V_1,\dots,V_k$ at the bottom and $n$ strands labeled $W_1,\dots,
W_n$ at the top is considered as a morphism from $V_1\otimes\dots\otimes
V_k\to W_1\otimes \dots\otimes W_n$. As usual, by default all strands are
oriented going from the bottom to top. Note that since $\C$ is assumed to
be a spherical category and not a braided one, no crossings are allowed in
the diagrams. 

For technical reasons, it is convenient to extend the graphical calculus by
allowing, in addition to rectangular coupons, also circular coupons
labeled with morphisms $\ph\in \<V_,1\dots, V_n\>$. This is easily seen to
be equivalent to the original formalism: every such circular coupon can be
replaced by the usual rectangular one as shown in \firef{f:round_coupon}.
\begin{figure}[ht]
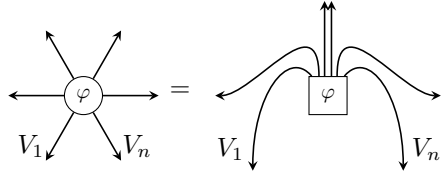

\tzRoundCouponI
=
\tzRoundCouponII
\caption{Round coupons}\label{f:round_coupon}
\end{figure}

We will also use the following convention: if a figure contains a pair of
circular coupons, one with outgoing edges labeled $V_1,\dots, V_n$ and
the other with edges labeled $V_n^*,\dots, V_1^*$ , and the coupons are 
labeled by pair of letters $\ph,\ph^*$ (or $\psi, \psi^*$, or \dots)
it will stand for summation over the dual bases:
\begin{equation}\label{e:summation_convention}
\tzSummationConventionI
\quad = \sum_\al\quad
\tzSummationConventionII
\end{equation}
where $\ph_\al\in \<V_1,\dots, V_n\>$, $\ph^\al\in \<V_n^*,\dots, V_1^*\>$
are dual bases with respect to pairing \eqref{e:pairing}.

The following lemma, proof of which is left to the reader,  lists some
properties of this pairing and its relation with the composition maps
\eqref{e:composition}. 

\begin{lemma}\label{l:pairing}
\par\noindent
\begin{enumerate}
 \item 
 If $X$ is simple and $\ph\in \<X,A\>$, $\ph'\in\<A^*,X^*\>$ then 
$$
\fig{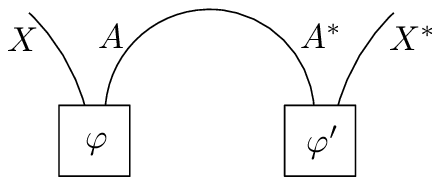}=\frac{(\ph,\ph')}{d_X}\quad \fig{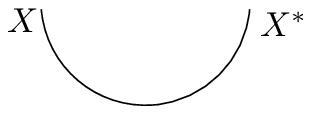}
$$
\item 
  $$
  \sum_{i\in\Irr(\C)} d_i 
  \tzPairingIII=\tzPairingIV
  $$
  \textup{(}we use here convention \eqref{e:summation_convention}\textup{)}. 
\item If $X$ is simple, $\ph\in \<A,
  X\>$, $\ph'\in \<X^*, A^*\>$,  $\psi\in \<X^*,B\>$, $\psi'\in
\<B^*,X\>$, then 
\begin{equation}\label{e:compatibility}
(\ph\cc{X} \psi,\psi'\cc{X^*}\ph')
=\frac{1}{d_X} (\ph,\ph')(\psi',\psi)
\end{equation}
\textup{(}see \firef{f:compatibility}\textup{)}.
\begin{figure}[ht]
$$\fig{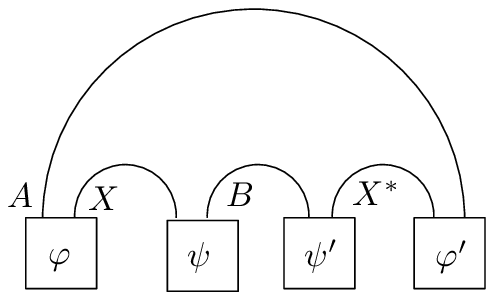}=\frac{(\ph,\ph')}{d_X}\fig{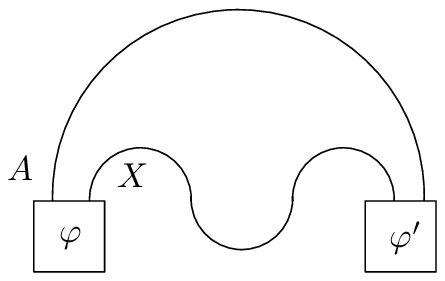}=\frac{(\ph,
\ph')(\psi',\psi)}{d_X}$$

\caption{Compatibility of pairing with composition.}
  \label{f:compatibility}
\end{figure}
\end{enumerate} 
\end{lemma}

\begin{corollary}\label{c:composition2}
Let $X$ be a simple object. Define the rescaled composition map
\begin{equation}\label{e:composition2}
\begin{aligned}
 \<V_1,\dots,V_n, X\>\otimes\<X^*, W_1,\dots,
W_m\>&\to\<V_1,\dots,V_n, W_1,\dots, W_m\>\\
\ph\otimes\psi&\mapsto \ph\ccc{X}\psi= \sqrt{d_X}\ \ev_X\circ
(\ph\otimes\psi)
\end{aligned}
\end{equation}
Then the rescaled composition map agrees with the pairing: 
$$(\ph\ccc{X} \psi,\psi'\ccc{X^*}\ph')= (\ph,\ph')(\psi',\psi)
$$ 
\textup{(}same notation as in \eqref{e:compatibility}\textup{)}. 
\end{corollary}

The following result, which easily follows from  \leref{l:pairing}, will
also be very useful.
\begin{lemma}\label{l:pairing2}
 If the subgraphs $A$, $B$ are not connected, then 
 $$
 \tzPairingV=\tzPairingVI
 $$
\end{lemma}

Finally, we will need the following result, which is the motivation for
the name ``spherical category''.

Let $\Ga$ be an oriented graph embedded in the sphere $S^2$, where each
edge $e$ is colored by  an object $V(e)\in \C$, and each vertex $v$ is
colored by a morphism $\ph_v\in \<V(e_1)^{\pm },\dots V(e_n)^{\pm} \>$,
where $e_1,\dots,e_n$ are the edges adjacent to vertex $v$, taken in
clockwise order, and $V(e_i)^{\pm}=V(e_i)$ if $e_i$ is outgoing
edge, and $V^*(e_i)$ if $e_i$ is the incoming edge.

By removing a point from $S^2$ and identifying $S^2\setminus{pt}\simeq
\R^2$, we can consider $\Ga$ as a planar graph. Replacing each vertex $v$
by a circular coupon labeled by morphism $\ph_v$ as shown in
\firef{f:sphere_diagram}, we get a graph of the type discussed above and
which therefore defines a number $Z_{RT}(\Ga)\in \kk$ (see, e.g.,
\ocite{BK} or \ocite{turaev}).

\begin{figure}[ht]
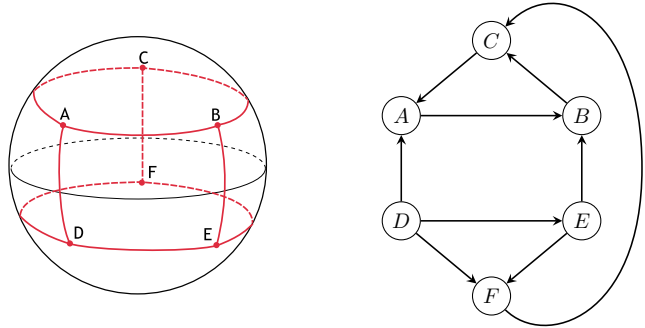

 \fig{figures/dual_graph2b}\qquad\qquad
 \tzSphereDiagram
\caption{Graph on a sphere and its ``flattening''
to the plane}\label{f:sphere_diagram}
\end{figure}

\begin{theorem}\ocite{barrett}\label{t:spherical1}
  The number $Z_{RT}(\Ga)\in \kk$ does not depend on the choice of a point 
  to remove from $S^2$ or on the choice of  order of edges at
  vertices compatible with the given cyclic order and thus defines an
  invariant of colored graphs on the sphere.
\end{theorem}

\section{Preliminaries II: Drinfeld Center}\label{s:center}

We will also need the notion of Drinfeld center of a spherical fusion
category. Recall that the Drinfeld center $Z(\C)$ of a fusion category $\C$
is defined as the category whose objects are pairs $(Y,\ph_Y)$, where $Y$
is an object of $\C$ and $\ph_Y$ -- a functorial isomorphism $Y\otimes -\to
-\otimes Y$ satisfying certain compatibility conditions (see
\ocite{muger1}).

As before, we will frequently use graphical presentation of morphisms
which involve objects both of $\C$ and $Z(\C)$. In these diagrams, we
will show objects of $Z(C)$ by double green lines   and the
half-braiding isomorphism $\ph_Y\colon Y\otimes V\to V\otimes Y$  by
crossing as in \firef{f:crossing}.
\begin{figure}[ht]
\tzCrossing
\caption{Graphical presentation of the half-braiding $\ph_Y\colon Y\otimes
V\to V\otimes Y$, $Y\in \Obj Z(\C)$, $V\in \Obj \C$}\label{f:crossing}
\end{figure}

We list here main properties of $Z(\C)$, all under the assumption that $\C$
is a spherical fusion category over an algebraically closed field of
characteristic zero.

\begin{theorem}{\ocite{muger2}}
 $Z(\C)$ is a modular category; in particular, it is
    semisimple with finitely many simple objects, it is braided and has
    a pivotal structure which coincides with the pivotal structure on
    $\C$. 
\end{theorem}

We have an obvious forgetful functor $F\colon Z(\C)\to \C$. To simplify
the notation, we will frequently omit it in the formulas, writing for
example $\Hom_\C(Y,V)$ instead of $\Hom_\C(F(Y),V)$, for $Y\in \Obj
Z(\C)$, $V\in \Obj \C$. Note, however, that if $Y,Z\in \Obj Z(\C)$, then
$\Hom_{Z(\C)}(Y,Z)$ is different from $\Hom_{\C}(Y,Z)$: namely,
$\Hom_{Z(\C)}(Y,Z)$ is a subspace in $\Hom_{\C}(Y,Z)$ consisting of those
morphisms that commute the with the half-braiding. The following lemma will
be useful in the future.

\begin{lemma}\label{l:projector}
  Let $Y,Z\in \Obj Z(\C)$. Define the operator 
  
  $P\colon \Hom_\C(Y,Z)\to
  \Hom_\C(Y,Z)$ by the following formula:
  $$
    P\psi=\frac{1}{\DD^2}\sum_{X\in \Irr(\C)} d_X\quad
    \tzProjectionI
  $$
  Then $P$ is a projector onto the subspace   
  $\Hom_{Z(\C)}(Y,Z)\subset   \Hom_\C(Y,Z)$. 
\end{lemma}
\begin{proof}It is immediate from the definition that if $\psi\in
  \Hom_{Z(\C)}(Y,Z)$, then $P\psi=\psi$. On the other hand, using
  \leref{l:pairing}, we get  that for any $\psi\in \Hom_\C(Y,Z)$, one has
  $(P\psi)\ph_Y=\ph_Z(P\psi)$:
  \begin{align*}
    (P\psi)\ph_Y &=
    \frac{1}{\DD^2}\sum_{j}d_{j}\quad\tzProjectionII=
  \frac{1}{\DD^2}\sum_{i,j}d_{i}d_{j}\quad \tzProjectionIII\\
  &=
\frac{1}{\DD^2}\sum_{i}d_i\quad \tzProjectionIV
  = \ph_Z(P\psi),
  \end{align*}
 (as before, we are using convention \eqref{e:summation_convention}). 
\end{proof}

The following theorem is a refinement of \ocite{ENO}*{Proposition 5.4}.
\begin{theorem}\label{t:I}
   Let $F\colon Z(\C)\to \C$ be the forgetful functor and $I\colon
  \C\to Z(\C)$ the (left) adjoint of $F$:
  $\Hom_{Z(\C)}(I(V),X)=\Hom_\C(V,F(X))$. Then for $V\in \Obj \C$, one
  has 
  \begin{equation}\label{e:I}
    I(V)=\bigoplus_{i\in \Irr(\C)} X_i\otimes V\otimes X_i^*
  \end{equation}
  with the half braiding given by
  \begin{figure}[ht]
    $$\bigoplus_{i,j\in \Irr (\C)} \sqrt{d_i}\sqrt{d_j} \quad
    \tzHalfBraid
    $$
    \caption{Half-braiding $I(V)\otimes
    W\to W\otimes I(V)$.}\label{f:I}
  \end{figure}
\end{theorem}
Note that instead of normalizing factor $\sqrt{d_i}\sqrt{d_j}$ we could
have used $d_i$ or $d_j$ ---  each of this would give an equivalent
definition.
\begin{proof}
Denote $Y=\bigoplus_{i\in \Irr(\C)} X_i\otimes V\otimes X_i^*$. It follows
from \leref{l:pairing} that the morphisms $Y\otimes W\to W\otimes Y$
defined by \firef{f:I} satisfy the compatibility relations required of half
braiding and thus define on $Y$ a structure of an object of $Z(\C)$. Now,
define for any $Z\in \Obj Z(\C)$, maps
\begin{align*}
\Hom_{Z(\C)}(Y,Z)&\to\Hom_\C(V,Z)\\
\Psi&\mapsto \Psi\circ P_0=\quad 
\tzFunctI
\end{align*}
where $P_0$ is the embedding $V=\one\otimes V\otimes \one\to Y=\bigoplus
X_i\otimes V\otimes X_i^*$
 and
\begin{align*}
\Hom_\C(V,Z)&\to\Hom_{Z(\C)}(Y,Z)\\
\Phi&\mapsto \bigoplus_{i\in\Irr(\C)}\sqrt{d_i} \tzFunctII
\end{align*}

It follows from \leref{l:projector} that these two maps are inverse to
each other. Composition in one direction is easy. First suppose 
$\Phi\in\Hom_\C(V,Z)$. The computation is shown below. 
$$
\Phi\to\bigoplus_{i}\sqrt{d_i}\quad \tzFunctII \to \Phi.
$$

The composition in opposite order is  as follows:
$$\Psi \to \tzFunctI
 \to
\bigoplus_{i}\sqrt{d_{i}} \tzFunctIII=
\bigoplus_{i} \sqrt{d_{i}} \sqrt{d_{j}} \tzFunctIV=\tzFunctV
$$

The first equality holds by functoriality of the half-braiding and
\firef{f:I}. The second equality is obvious. 
Therefore, the two maps are inverses to one another and we have
$\Hom_{Z(\C)}(Y,Z)=\Hom_\C(V,Z)$; thus, $Y=I(V)$. 
\end{proof}

An easy generalization of \thref{t:spherical1} allows us to consider
graphs in which some of the edges are labeled by objects of $Z(\C)$.

Let $\hat\Ga$ be a graph which consists of a usual graph $\Ga$ embedded in
$S^2$ as in \thref{t:spherical1} and a finite collection of
non-intersecting oriented  arcs  $\ga_i$ such that endpoints of each arc
$\ga$  are vertices of graph $\Ga$, and each vertex has a neighborhood in
which arcs $\ga_i$ do not intersect edges  of $\Ga$; however, arcs $\ga_i$
are allowed to intersect edges of $\Ga$ away from vertices. Note that this
implies that for each vertex $v$, we have a natural cyclic
order on the set of all edges of $\hat\Ga$ (including arcs  $\ga_i$)
adjacent to $v$.

Let us color such diagram, labeling each edge of $\Ga$ by an object of
$\C$, each arc $\ga$  by an object of $Z(\C)$, and each vertex $v$ by a
vector $\ph_v\in \<V^\pm (e_1), \dots, V^{\pm}(e_n)\>$ where $e_1,\dots,
e_n$ are edges of $\hat\Ga$ adjacent to $v$ (including the arcs
$\ga_i$), and the signs are chosen as in \thref{t:spherical1}.

\begin{figure}[ht]
 \fig{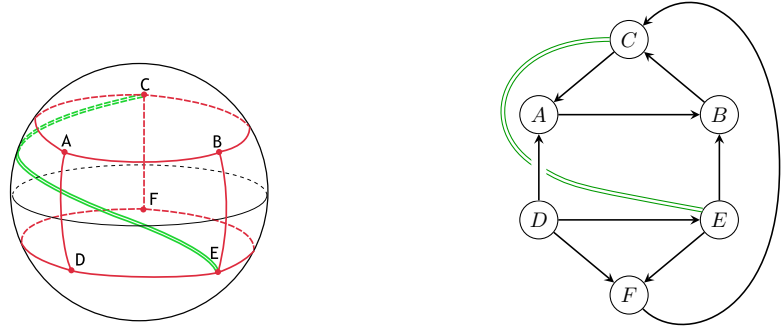}\qquad
\tzSphereDiagramII
\caption{Diagram $\hat\Ga$ on the sphere and its
flattening to the plane. Arc $\ga$ is shown
by a double line.}\label{f:sphere_diagram2}
\end{figure}

As before, by removing a point from $S^2$ and choosing a linear order of
edges (including the arcs) at every vertex, we get a diagram in the plane;
however, now the projections of  arcs $\ga_i$ can intersect edges of $\Ga$
as shown in \firef{f:sphere_diagram2}. Let us turn this into a tangle
diagram by replacing each intersection by a picture where the arch $\ga_i$
goes under the edges of $\Ga$, as shown in \firef{f:sphere_diagram2}.

Such a diagram defines a number $Z_{RT}(\hat\Ga)$ defined in the usual way,
with the extra convention shown in \firef{f:crossing}.

\begin{theorem}\label{t:spherical2}
  The number $Z_{RT}(\hat\Ga)\in \kk$ does not depend on the choice of a
  point to remove from $S^2$ orand thus defines an
  invariant of colored graphs on the sphere. Moreover, this number is
  invariant under homotopy of arcs $\ga_i$. 
\end{theorem}
\begin{proof}
The fact that it is independent of the choice of point to remove and thus
is an invariant of a graph on the sphere immediately follows from
\thref{t:spherical1}: replacing every crossing by a coupon colored by
half-braiding $\ph_Y$ gives a graph as in \thref{t:spherical1}. Invariance
under homotopy of arcs $\ga$ follows from compatibility conditions on
half-braiding shown in \firef{f:homotopy}.
\begin{figure}[ht]
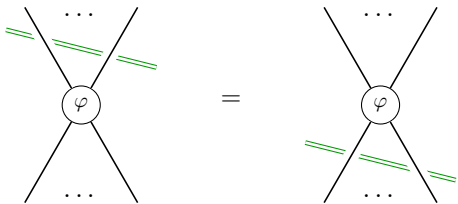

\qquad \tzHomotopyI\qquad =\qquad \tzHomotopyII
\caption{Invariance under homotopy}\label{f:homotopy}
\end{figure}

\end{proof}

Finally, we will need one more useful construction.

For any  $Y\in \Obj Z(\C)$, we define a functor
functor $\C^{\boxtimes n}\to \Vect$ by
\begin{equation}\label{e:vev-Z}
 \<V_1,\dots,V_n\>_Y=\Hom_\C(\one,Y\otimes  V_1\otimes\dots\otimes V_n)
\end{equation}
for any collection $V_1,\dots, V_n$ of objects of $\C$. As before, we have
functorial isomorphisms 
\begin{equation}\label{e:cyclic-Z}
 z_Y\colon\<V_1,\dots,V_n\>_Y\simeq \<V_n, V_1,\dots,V_{n-1}\>_Y
\end{equation}
obtained as composition 
$$
\<Y,V_1,\dots,V_n\>\to \<V_n, Y, V_1,\dots,V_{n-1}\>\to 
\<Y, V_n, V_1,\dots,V_{n-1}\>
$$
(the first isomorphism is the cyclic isomorphism \eqref{e:cyclic}, the
second one is the inverse of half-braiding $\ph_Y$). Note however that
in general  we do not have $z_Y^n=\id$. 

\section{Polytope decompositions}\label{s:polytope}

It will be convenient to rewrite the definition of Turaev--Viro (TV) 
invariants using not just triangulations, but more general cellular
decompositions. 
In this section we give precise definitions of these decompositions.

In what follows, the word ``manifold'' denotes a compact, oriented,
piecewise-linear (PL) manifold; unless otherwise specified, we assume that
it has no boundary. Note that in dimensions 2 and 3, the category of PL
manifolds is equivalent to the category of topological manifolds. For an 
oriented manifold $M$, we will denote by $\ov{M}$ the same manifold with
opposite orientation, and by $\del M$, the boundary of $M$ with induced
orientation. 

Instead of triangulated manifolds as in \ocite{barrett}, we prefer to
consider more general cellular decompositions, allowing individual
cells to be arbitrary polytopes (rather than just simplices); moreover,  
we will allow the attaching maps to identify some of the boundary points,
for example gluing polytopes so that some of the vertices coincide.
On the other hand, we do not want to consider arbitrary cell
decompositions (as is done, say, in \ocite{oeckl}), since it would make
describing the  elementary moves between two such decompositions more
complicated.  The following definition is the compromise; for lack of a
better word,   we will call such decompositions {\em polytope
decompositions}.

Recall that a cellular decomposition of a manifold $M$ is a collection 
of inclusion maps $B^d\to M$, where $B^d$ is the (open) $d$-dimensional
ball, satisfying certain conditions. Equivalently, we can replace
$d$-dimensional balls with $d$-dimensional cubes $I^d=(0,1)^d$. For a PL
manifold, we will call such a cellular decomposition a PL  decomposition 
if each inclusion map  $(0,1)^d\to M$ is a PL map. In particular, every
triangulation of a PL  manifold gives such a cellular decomposition (each
$d$-dimensional simplex is PL homeomorphic  to a $d$-dimensional cube).  

We will call a cell {\em regular} if the corresponding map $(0,1)^d\to M$
extends to a map of the closed cube $[0,1]^d\to M$ which is a
homeomorphism onto its image. 

\begin{definition}\label{d:polytope_d}
 A polytope decomposition of a  2- or 3-dimensional PL manifold $M$
(possibly with boundary) is a cellular   decomposition which can be
obtained from a triangulation by a sequence of moves 
 M1---M3 below (for $\dim M=2$, only moves M1, M2).
 
\begin{description}
 \item[M1: removing a vertex]
   Let $v$ be a vertex which has a neighborhood whose intersection with 
   the 2-skeleton is homeomorphic to the ``open book'' shown below with
    $k\ge 1$     leaves;    moreover, assume that all leaves in the
    figure are distinct 2-cells and the two
    1-cells  are also  distinct (i.e., not two ends of the same
    edge). Then move M1 removes vertex $v$ and replaces two 1-cells
  adjacent to it with a single 1-cell.
   \begin{figure}[ht]
   $$\figscale{0.7}{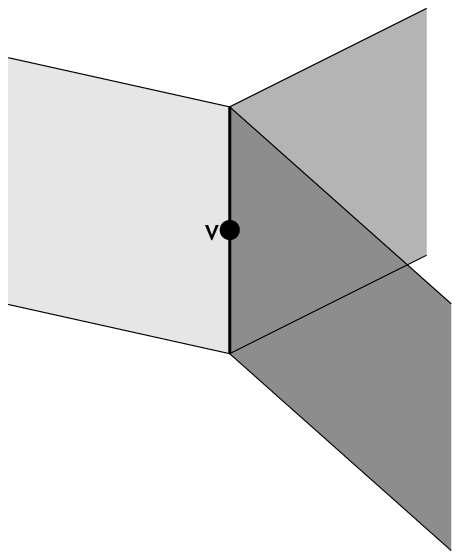}
     \xxto{\text{remove vertex } v}
     \figscale{0.7}{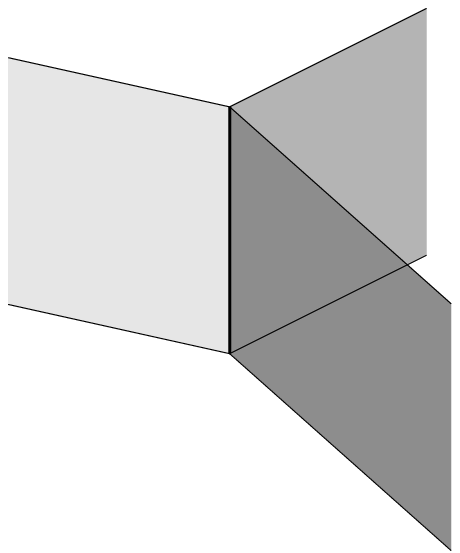}
   $$
  \caption{Move M1}\label{f:movef1}
  \end{figure}

 \item[M2: removing an edge]
   Let $e$ be a 1-cell which is regular and which is 
   adjacent to exactly two distinct 2-cells $c_1, c_2$ as shown in the
    figure below.  Then the move M2 removes the edge $e$ and replaces
  the cells $c_1,c_2$ with a  single cell $c$. 
   \begin{figure}[ht]
      $$\figscale{0.7}{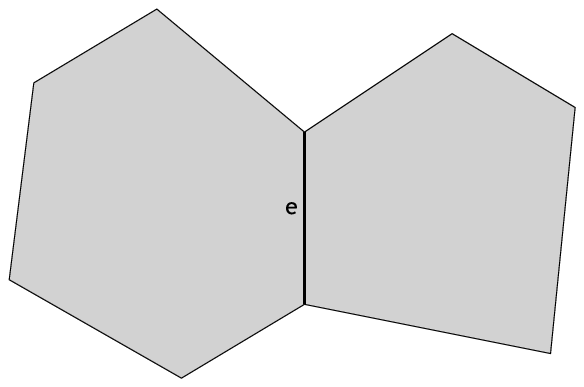}
     \xxto{\text{remove edge } e}
     \figscale{0.7}{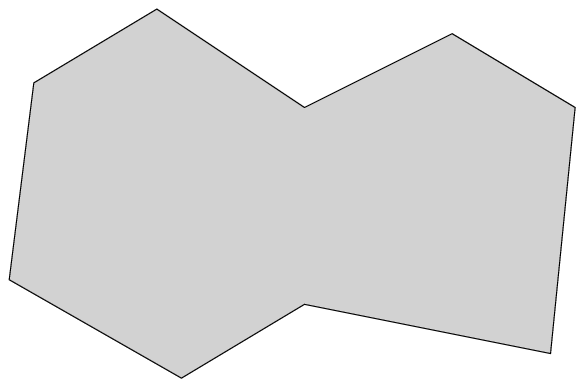}
   $$
  \caption{Move M2}\label{f:movef2}
  \end{figure}

 \item[M3: removing a 2-cell]
   Let $c$ be a 2-cell which is regular  and which is 
   adjacent to exactly two distinct 3-cells $F_1, F_2$ as shown in the
  figure below.  Then the move M2 removes the 2-cell $c$ and replaces the
  cells $F_1,F_2$ with a single cell $F$.
   \begin{figure}[ht]
   $$\figscale{0.7}{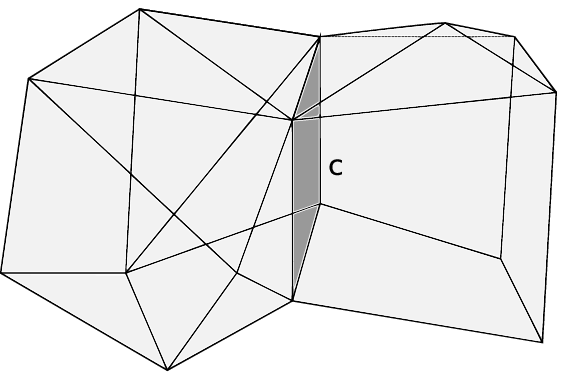}
     \xxto{\text{remove face } c}
     \figscale{0.7}{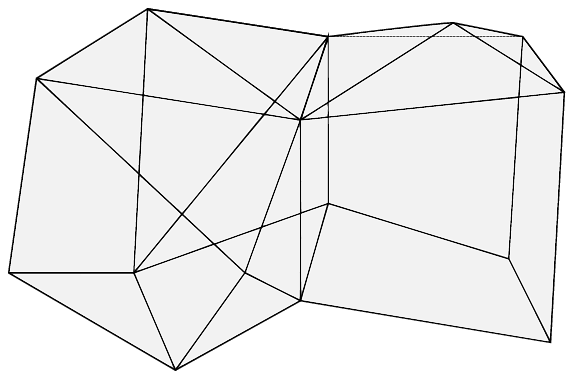}
   $$
  \caption{Move M3}\label{f:movef3}
  \end{figure}

\end{description}

A 2 or 3-dimensional PL manifold $M$ with boundary together with a choice
of polytope decomposition will be called a {\em combinatorial manifold};
for $\dim M=2$, we will also use the term ``combinatorial surface''. We
will use  script letters to denote combinatorial manifolds and Roman
letters for underlying PL manifolds.  
\end{definition}

Note that the extension of the  inclusion maps $(0,1)^d\to M$ to the 
boundary does not have to be injective. 

If $F$ is an oriented $d$-dimensional cell of a combinatorial manifold $\M$
(i.e., a pair consisting of a cell and its orientation), we can define
its boundary $\del F$ in the obvious way, as a formal union of oriented
$(d-1)$-dimensional cells. Note that $\del F$ can contain the same 
(unoriented) cell $C$ more than once: for example, one could have 
$\del F=\dots \cup C\cup \ov{C}\dots$. 
\begin{lemma}\label{l:delF}
If $\M$ is a combinatorial manifold  of dimension $d$ with boundary, then 
$$
\bigcup_{F}\del F=\Bigl(\bigcup_{C\in \del M} C\Bigr)
   \cup
     \Bigl(\bigcup_{c_{in}} c'_{in}\cup c''_{in}     \Bigr)
$$
where $F$ runs over the set of $d$-cells of $M$ \textup{(}each taken with
induced orientation\textup{)}, $C$ runs over the set of $(d-1)$-cells of
$\del M$ \textup{(}each taken with induced orientation\textup{)}, and
$c_{in}$ runs over the set of \textup{(}unoriented\textup{)} $(d-1)$-cells
in the interior of $M$, with $c', c''$ denoting two possible orientations
of $c$ \textup{(}so that $\ov{c'}=c''$\textup{)}. 
\end{lemma}

The main result of this section is the following theorem.

\begin{theorem}\label{t:moves}
  Let $M$ be a PL 2- or 3-manifold without boundary. Then any two
  polytope decompositions of $M$ can be obtained from each other by a
  finite sequence of moves M1--M3 and their inverses \textup{(}if $\dim
  M=2$, only     moves M1, M2 and their inverses\textup{)}.
\end{theorem}
\begin{proof}
 It is immediate from the definition that it suffices to prove that any
  two triangulations can be obtained one from another by a sequence of
  moves M1--M3 and their inverses. On the other hand, since it is 
  known that any two triangulations  are related by a sequence of Pachner
  bistellar moves \ocite{pachner}, it suffices to show that each Pachner
  bistellar move can be  presented as a  sequence of moves M1--M3 and
  their inverses. For
  $\dim M=2$,   this is left as an easy exercise to the reader; for $\dim
  M=3$, this is   shown in \firef{f:pachner1}, \firef{f:pachner2}.
   \begin{figure}[ht]
   \begin{align*}&\figscale{0.7}{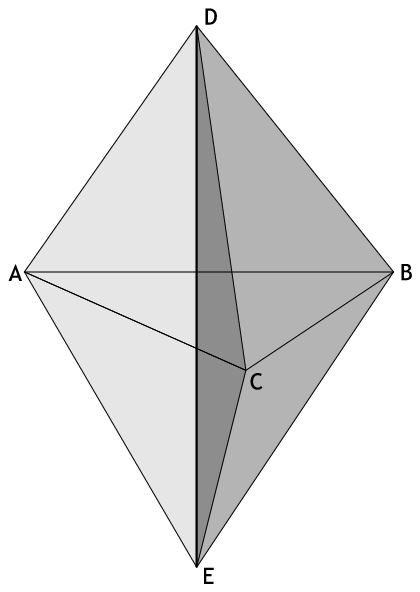}
     \xxto{\text{remove face } ADE}
    \figscale{0.7}{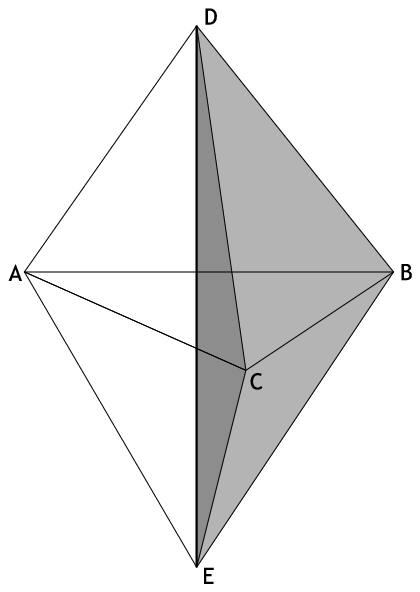}
     \xxto{\text{remove edge } DE}
     \figscale{0.7}{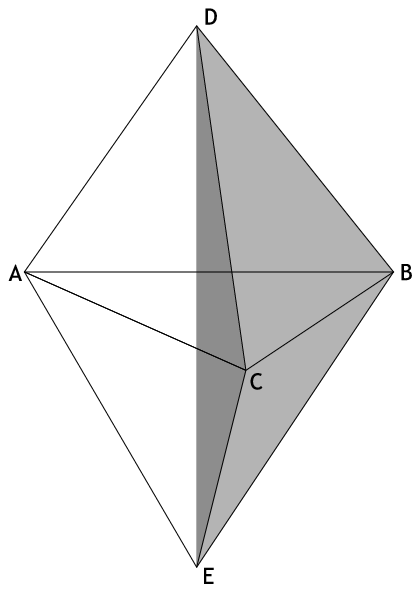}
     \\ 
   &  \xxto{\text{remove face } CDBE}  
   \figscale{0.7}{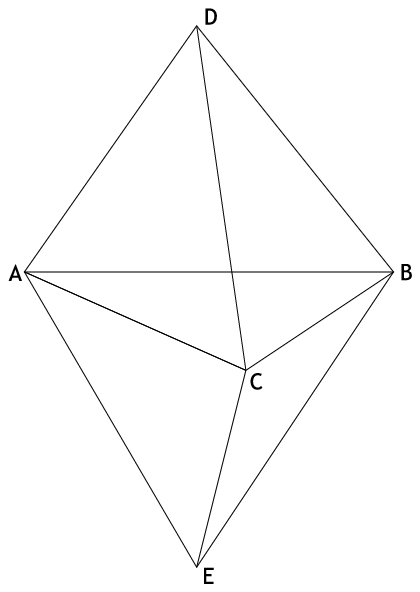} 
    \xxto{\text{add face }ABC}
      \figscale{0.7}{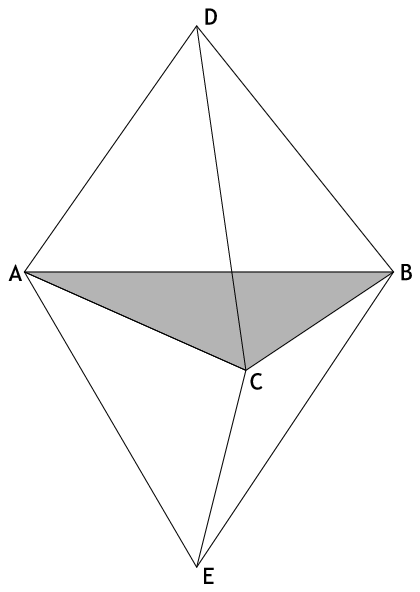}
    \end{align*}
  \caption{Pachner 3-2 move as composition of elementary
            moves}\label{f:pachner1}
  \end{figure}

   \begin{figure}[ht]
    \begin{align*}&\figscale{0.7}{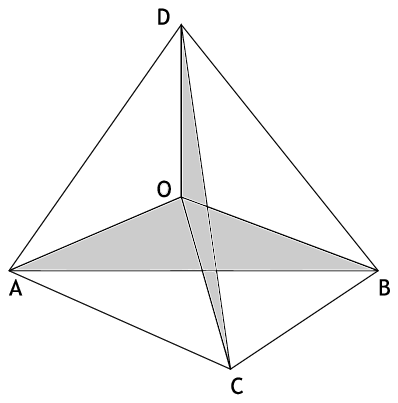}
       \xxto{\text{remove faces } AOB, DOC}
       \figscale{0.7}{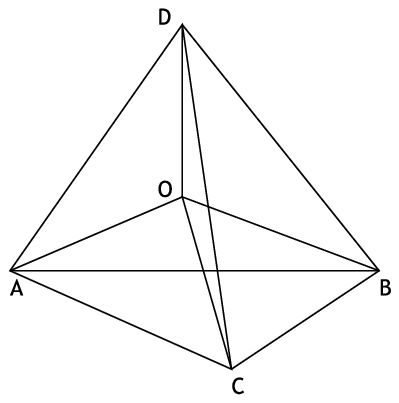}
    \\ 
   & \qquad \xxto{\text{remove edges } AO, OB}
     \figscale{0.7}{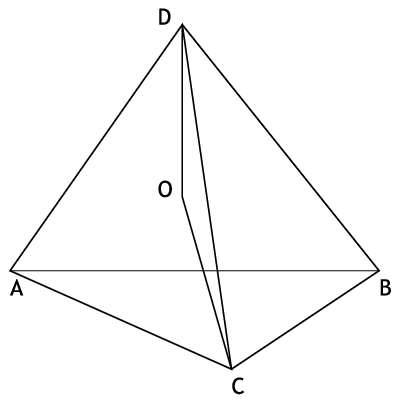}
      \xxto{\text{remove vertex } O}  
   \figscale{0.7}{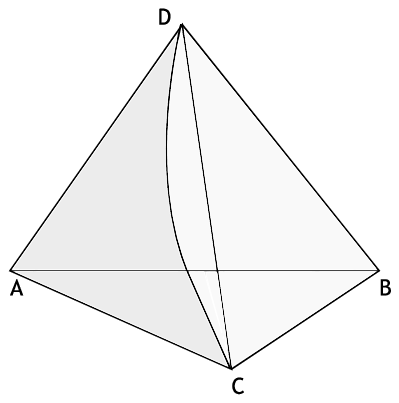} 
   \\
    & \qquad \xxto{\text{remove edge } DC}
      \figscale{0.7}{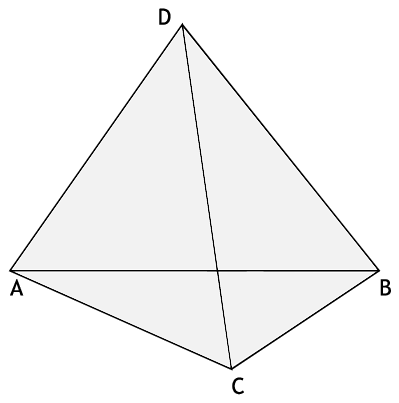}
      \xxto{\text{remove face } ADBC}
     \figscale{0.7}{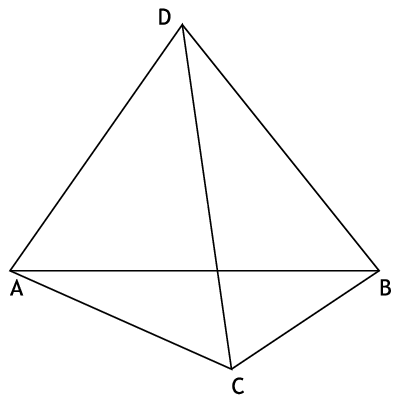}
    \end{align*}
      
\caption{Pachner 4-1 move as composition of elementary
            moves}\label{f:pachner2}
  \end{figure}

\end{proof}

This can be generalized to manifolds with boundary.

\begin{theorem}\label{t:moves_rel}
  Let $M$ be a PL 2- or 3-manifold with boundary and let $\N$ be a
  polytope decomposition of $\del M$. Then
  \begin{enumerate}
  \item $\N$ can be extended to a polytope decomposition $\M$ of $M$.
  \item Any two polytope decompositions $\M_1, \M_2$ of $M$ which coincide
  with $\N$ on  $\del M$ can be obtained from each other by  a finite
  sequence of moves M1--M3 and their  inverses which do not change the
  polytope decomposition of $\del M$.
  \end{enumerate}
\end{theorem}
\begin{proof}
 The theorem immediately follows from the following two lemmas. 
 \begin{lemma}\label{l:moves_rel1}
  If $\N$ is a triangulation, then the statement of the theorem holds. 
 \end{lemma}
 \begin{lemma}\label{l:moves_rel2}
  If $\N$ is obtained from another polytope decomposition $\N'$ of $\del 
  M$ by a move M1, M2 \textup{(}only M1 if $\dim M=2$\textup{)}, and the
  statement of the theorem holds for $\N'$, then the statement of the
  theorem holds for $\N$.
 \end{lemma}
\begin{proof}[Proof of \leref{l:moves_rel1}]
  Follows from the relative version of Pachner moves \ocite{casali}.
\end{proof}
 \begin{proof}[Proof of \leref{l:moves_rel2}]

  We will do the proof in the case when $\dim M=3$ and $\N$ is obtained
  from $\N'$ by erasing an edge $e$ separating two cells $c_1, c_2$.  The
  proof in other cases is   similar and left to the reader. 

  Let $\M'$ be a polytope decomposition of the $M$ which agrees with
  $\N'$ on $\del M$; by assumption such a decomposition exists. Denote
  $c=c_1\cup e\cup c_2$. Let us glue to $\M'$ another copy of 2-cell $c$
  along the boundary of $c_1\cup e\cup c_2$ and a 3-cell $F$ filling the
  space between $c_1\cup e\cup c_2$ and $c$ as   shown in
  \firef{f:rel_moves2} 

   \begin{figure}[ht]
  $$
  \fig{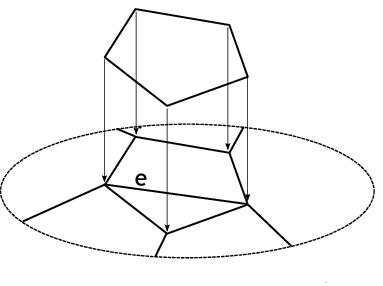}\xxto{}\fig{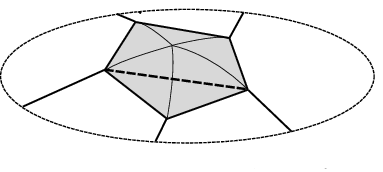}
  $$
\caption{Proof of \leref{l:moves_rel2}}\label{f:rel_moves2}
  \end{figure}
  This gives a new
  manifold $\tilde M$ which is obviously homeomorphic to $M$, together
  with a polytope decomposition $\tilde M$ such that its restriction to the
  boundary is $\N$. This proves existence of extension. Moreover, it is
  immediate from the assumption on $\N'$ that any two polytope
  decompositions $\tilde M_1$, $\tilde M_2$ obtained in this way from
  polytope decomposition $\M'_1, \M'_2$  extending $\N'$  can be
  obtained from each other by a sequence of moves  M1,  M2 and their
  inverses which do not  change decomposition of $\del  M$.
  
  To prove the second part, let $\M_1$, $\M_2$ be two polytope
  decompositions which coincide  with $\N$ on $\del M$. Let us add 2-cells
  $c_1,c_2$ and an edge $e$ to to each of these
    decomposition as shown in \firef{f:rel_moves2b}; this  gives  new
  decompositions
  $\tilde \M_1,\tilde \M_2$ which are of  of the form discussed  above and
    thus can be obtained from each   other by  a sequence of moves
  M1,  M2 and their inverses which do not  change decomposition of $\del
  M$. 
   \begin{figure}[ht]
$$   \fig{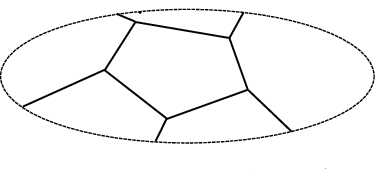}\xxto{}\fig{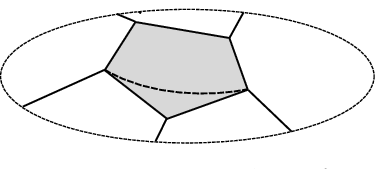}$$
  \caption{Proof of \leref{l:moves_rel2}}\label{f:rel_moves2b}
  \end{figure}
  
 \end{proof}
\end{proof}

Finally, we will need a slight generalization of this result.

\begin{theorem}\label{t:moves_rel2}
  Let $M$ be a  3-manifold with boundary and let $X\subset \del M$ be a
  subset homeomorphic to a 2-manifold with boundary. Let $\N$ be a
  polytope decomposition of a $X$. Then
  \begin{enumerate}
  \item $\N$ can be extended to a polytope decomposition $\M$ of $M$
  \item Any two polytope decompositions $\M_1, \M_2$ of $M$ which coincide
  with $\N$ on  $X$ can be obtained from each other by  a finite
  sequence of moves M1--M3 and their  inverses which do not change the
  polytope decomposition of $X$.
  \end{enumerate}
\end{theorem}
A proof is similar to the proof of the previous theorem; details are left
to the reader. 
\section{TV invariants from polytope decompositions}\label{s:TV}

In this section, we recall the definition of Turaev--Viro (TV) invariants
of 3-manifolds. Our exposition essentially follows the approach of
Barrett and Westbury \ocite{barrett}; however, instead of triangulations we
use more general polytope decompositions as defined in the previous
section.

Let $\C$ be a spherical fusion category as in \seref{s:prelim}, and $\M$
--- a combinatorial 3-manifold. We denote by $E$ the set of oriented edges
(1-cells) of $\M$. Note that each 1-cell of $\M$ gives rise to two oriented
edges, with opposite orientations. 

\begin{definition}\label{e:labeling}
 An  {\em labeling} of $\M$ is a map $l\colon E\to \Obj \C$ which assigns
  to every oriented edge $e$ of $\M$ an object $l(e)\in \Obj \C$ such that
  $l(\ov{e})=l(e)^*$. A labeling is called simple if for every edge,
  $l(e)$ is simple.  
 
  Two labelings are called equivalent if $l_1(e)\simeq l_2(e)$ for every
  $e$. 
\end{definition}

Given a combinatorial 3-manifold $\M$ and a labeling $l$, we
define, for every oriented 2-cell $C$, the state space
\begin{equation}\label{e:H(C)}
 H(C,l)=\<l(e_1), l(e_2),\dots, l(e_n)\>,\qquad 
 \del C=e_1\cup e_2\dots\cup e_n
\end{equation}
where the edges $e_1,\dots, e_n$ are taken in the counterclockwise order on
$\del C$ as shown in \firef{f:state_space1}.
  \begin{figure}[ht]
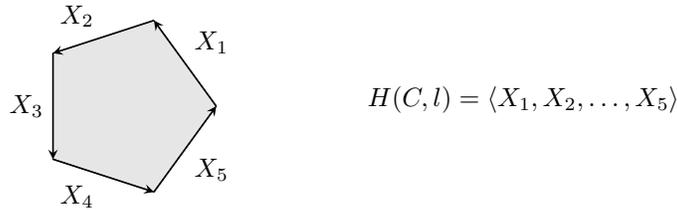

  \tzStateSpaceI
  \hspace{1.5cm} $H(C,l)=\<X_1,X_2,\dots, X_5\>$
  \caption{Defining the state space for a 2-cell}\label{f:state_space1}
  \end{figure}

Note that by \eqref{e:cyclic}, up to a canonical isomorphism, the state
space only depends  on the cyclic order of $e_1,\dots, e_n$ (which is
defined by $C$) and does not  depend on the choice of the starting point.

If $\N$ is an oriented 2-dimensional combinatorial manifold, we define the
state space
$$
H(\N,l)=\bigotimes_C H(C,l)
$$
where the product is over all 2-cells $C$, each taken with 
orientation induced from orientation of $\N$. 

Finally, we define
\begin{equation}\label{e:H(N)}
 H(\N)=\bigoplus_{l}H(\N,l),
\end{equation}
where the sum is over all simple labelings up to equivalence.

In the case when $\N$ is a triangulated surface, this definition
coincides with the one in \ocite{barrett}.

Note that it is immediate from \eqref{e:dual} that we have canonical
isomorphism
\begin{equation}\label{e:H-duality}
H(\ov{\N})=H(\N)^*.
\end{equation}

Next, we define the TV invariant of 3-manifolds. Let $\M$ be a
combinatorial 3-manifold with boundary. Fix a labeling $l$ of edges of
$\M$. Then every  3-cell $F$ defines a vector
$$
Z(F,l)\in  H(\del F,l)
$$
defined as follows. Recall that $F$ is an inclusion 
$F\colon (0,1)^3\to M$. The pullback of the polytope  decomposition of
$\M$  gives a polytope decomposition of $\del (0,1)^3\simeq S^2$. Consider
the dual  graph $\Ga$ of this decomposition and  choose an
orientation for every edge of this dual graph (arbitrarily) as shown in
\firef{f:dual_graph}.

   \begin{figure}[ht]
   $$\fig{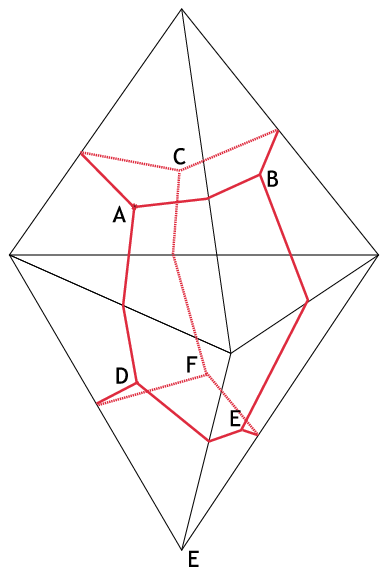}\xxto{}\fig{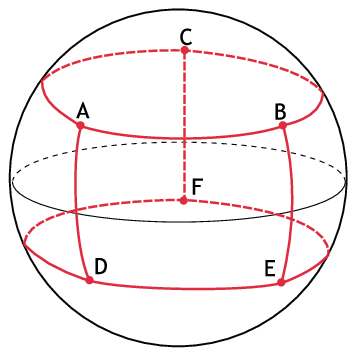}$$
  \caption{The dual graph on the boundary of a 3-cell}\label{f:dual_graph}
  \end{figure}

Note that a labeling $l$ of $\M$ defines  a labeling of edges of this dual
graph as shown in \firef{f:state_space2}. Moreover, choose, for every face
$C\in \del F$, an element $\ph_C\in H(C,l)^*=\<l(e_n)^*,\dots, l(e_1)^*\>$.
Then this collection of morphisms defines a coloring of  vertices of
$\Ga$.

   \begin{figure}[ht]
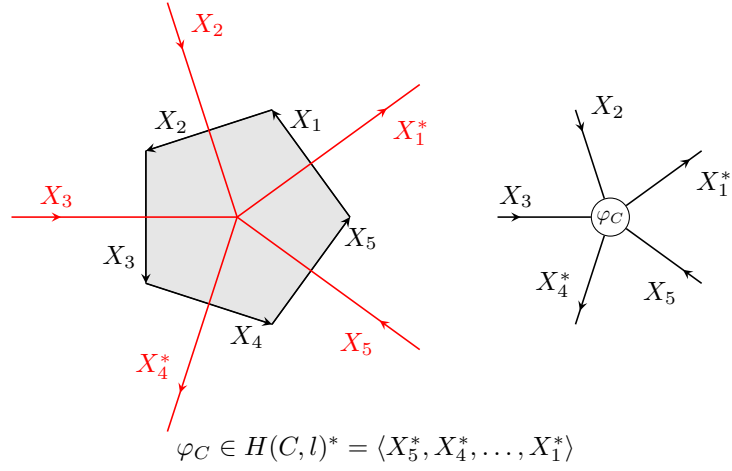

   \tzStateSpaceII
  \qquad $\ph_C \in H(C,l)^*=\<X_5^*,X_4^*, \dots, X_1^*\>$
  \caption{Coloring of the dual graph}\label{f:state_space2}
  \end{figure}

By \thref{t:spherical1}, we get an invariant $Z_{RT}(\Ga)\in \kk$, which
depends on the choice of labeling of edges $l$ and on the choice of
morphisms $\ph_C$.  We define $Z(F,l)\in \otimes_{C}H(C,l)$
by 
\begin{equation}\label{e:Z(F)}
(Z(F,l), \otimes \ph_C)=Z_{RT}(\Ga, l, \{\ph_C\}).
\end{equation}

Again, if $F$ is a tetrahedron, then this coincides with the definition in
\ocite{barrett}; if $\C$ is the category of representations of
quantum $\mathfrak{sl}_2$, these numbers are the $6j$-symbols. 

We can now give a definition of the TV invariants of combinatorial
3-manifolds.

\begin{definition}\label{d:TV_invariant}
  Let $\M$ be a combinatorial 3-manifold with boundary and $\C$ -- a
  spherical category. Then for any coloring $l$, define a vector 
$$
Z_{TV}(\M,l)\in  H(\del \M,l)
$$
by 
$$
Z_{TV}(\M,l)=\ev\Bigl(\bigotimes_F Z(F,l)\Bigr)
$$
where 
\begin{itemize}
  \item $F$ runs over all 3-cells in $M$,   each taken with the induced
    orientation, so that  
    $$
      \bigotimes_F Z(F,l)\in \bigotimes_F H(\del F,l)=H(\del \M,l)\otimes
      \bigotimes_{c}  H(c',l)\otimes H(c'',l)
    $$
    (compare with \leref{l:delF})  

  \item $c$ runs over all unoriented 2-cells in
    the interior of   $M$, $c', c''$ are the two orientations of such a
    cell, so that   $c'=\ov{c''}$. 
  \item $\ev$ is the tensor product over all $c$ of evaluation maps
    $H(c',l)\otimes  H(c'',l)=H(c',l)\otimes H(c',l)^*\to \kk$
\end{itemize}

Finally, we define
$$
Z_{TV}(\M)=\DD^{-2v(\M)}\sum_{l} \Bigl(
Z_{TV}(\M,l)\prod_{e}d^{n_e}_{l(e)}\Bigr)
$$
where 
\begin{itemize}
\item the sum is taken over all equivalence classes of simple labelings
  of $\M$,
\item $e$ runs over the set of all (unoriented) edges of $\M$
\item $\DD$ is the dimension of the category $\C$ (see \eqref{e:DD}), and
$$
v(\M)=\text{number of internal vertices of } \M+\frac{1}{2}\text{(number of
vertices on }\del \M)
$$
\item $d_{l(e)}$ is the categorical dimension of $l(e)$ and 

$$n_e=\begin{cases}1, &\quad e \text{ is an internal edge}\\
                   \tfrac{1}{2}, &\quad e \in \del \M
                   \end{cases}
$$
\end{itemize}
\end{definition}

It is easy to see that in the special case of triangulated manifold, this
coincides with the construction in \ocite{barrett}.

\begin{theorem}\label{t:main1}
  If $M$ is a PL manifold without boundary, then the
  number $Z_{TV}(M)\in \kk$ defined in \deref{d:TV_invariant} does not
  depend on the choice of polytope decomposition of $M$: for any two
  choices of polytope decomposition, the resulting 
  invariants are equal.
\end{theorem}
The proof of this theorem will be given in \seref{s:proof}.

These invariants can be extended to a TQFT. Namely, let $\M$ be a
combinatorial 3-cobordism between two 2-dimensional combinatorial
manifolds $\N_1,\N_2$, i.e a combinatorial manifold $\M$ with boundary such
that $\del \M=\ov{\N_1}\sqcup \N_2$ (note that the combinatorial structure
on $M$ automatically defines a combinatorial structure on $\del M$). Then
$H(\del \M)=H(\N_1)^*\otimes H(\N_2)=\Hom_\kk(H(\N_1),H(\N_2))$, so
\deref{d:TV_invariant} defines an element $Z(\M)\in
\Hom_\kk(H(\N_1),H(\N_2))$, i.e. a linear operator 
$$
Z(\M)\colon H(\N_1)\to H(\N_2).
$$

\begin{theorem}\label{t:main2}
\par\noindent
  \begin{enumerate}
      \item So defined invariant satisfies the gluing axiom: if $\M$ is a
        combinatorial 3-manifold with boundary $\del \M=\N_0\cup
        \N\cup\ov{\N}$, and $\M'$ is the manifold obtained by identifying
        boundary components $\N,\ov{\N}$ of $\del\M$ with the obvious cell
        decomposition, then  we have
    $$
      Z_{TV}(\M')=\ev_{H(\N)}Z_{TV}(\M)=\sum_\al (Z_{TV}(\M),
          \ph_\al\otimes \ph^\al),
    $$
        where $\ev$ is the evaluation map $H(\N)\otimes H(\ov{\N})\to
        \kk$, and $\ph_\al\in H(\N)$,  $\ph^\al\in H(\ov{\N})$ are dual
        bases. 
 
      \item If a $M$ is a 3-manifold with boundary, and $\M', \M''$ are two
        polytope decompositions of $M$  which agree     on the
        boundary, then
        $Z(\M')=Z(\M'')\in H(\del \M')=H(\del \M'')$. 
        
    \item For a combinatorial 2-manifold $\N$, define $A_{\N}\colon H(\N)\to
      H(\N)$ by 
      \begin{equation}\label{e:projector}
       A_{\N}=Z_{TV}(\N\times I)
      \end{equation}
       Then $A_{\N}$ is a projector: $A_{\N}^2=A_{\N}$.

    \item For a combinatorial 2-manifold $\N$, define the vector space
    \begin{equation}\label{e:modular_functor}
    Z_{TV}(\N)=\mathrm{Im}(A_{\N}\colon H(\N)\to H(\N))
    \end{equation}
  where $A$ is the projector \eqref{e:projector}.
  Then the space $Z_{RT}(N)$ is an invariant of PL manifolds: if $\N',
  \N''$ are two different polytope decompositions of the same PL manifold
  $N$, then one has a canonical isomorphism $Z(\N')\simeq Z(\N'')$.
   
  \item The assignments $N\mapsto Z_{TV}(N)$, $M\mapsto Z_{TV}(M)$ give a
  functor from the category of PL 3-cobordisms to the category of
  finite-dimensional vector spaces and thus define a $2+1$-dimensional
  TQFT.
\end{enumerate}

\end{theorem}
\begin{proof}
   Part (1) is immediate from the definition. 

   Part (2) will be proved in \seref{s:proof}.

    To prove part (3), note that
    gluing of two cylinders again gives a cylinder, so (3) follows from
    (1) and (2). 
  
  To prove (4), let $\N', \N''$ be two different polytope decompositions of
  $N$. Consider the cylinder $C=N\times I$ and choose a polytope
    decomposition of $C$ which agrees with $\N'$ on $N\times\{0\}$ and
    agrees with $\N''$ on $N\times\{1\}$ (existence of such a
      decomposition follows from \thref{t:moves_rel}).  Consider the
    corresponding operator 
    $F_1=Z(C)\colon  H(\N')\to H(\N'')$. In a similar way, define an
    operator $F_2\colon 
    H(\N'')\to H(\N')$. Then it follows from (2) that $F_1F_2=A_{\N''}$,
    and $F_2F_1=A_{\N'}$. Thus, $F_1, F_2$ give rise to mutually
    inverse isomorphisms $Z_{TV}(\N')\to Z_{TV}(\N'')$.

   Part (5) follows immediately from (1)--(4).
    
\end{proof}

Note that in the PL category, gluing along a boundary component is well
defined: gluing together PL manifolds results canonically in a PL
manifodl (unlike the smooth category). 

\begin{example}\label{x:H(S)}
  Let $G$ be a finite group and $\C=\VecG$ --- the category of $G$-graded
  vector spaces, with obvious  tensor structure. Then a simple labeling is
  just labeling of edges of $\M$ with elements of the group $G$, and for a
  2-cell $C$, we have
  $$
  H(C,l)=\begin{cases} \kk, &\prod_{\del C} l(e)=1\\
                      0, & \text{otherwise}
          \end{cases}
  $$
  Thus, we see that in this case the state space $H(\N)$ is the space of
  flat $G$---connections (which depends on the choice of polytope
  decomposition!). It is well-known that in this case the projector
  $A=Z_{TV}(\Si\times I)$  is the operator of averaging over the action of
  the gauge group $G^{v(\N)}$, where $v(\N)$  is the set of vertices of $\N$. 
  Thus the space $Z(N)$ is the space of gauge equivalence classes of $G$--connections.
\end{example}
\begin{example}\label{x:2sphere}
We verify $Z_{TV}(S^2) =\kk$ as is required by the definition of a TQFT.
We pick the polytope decomposition of $S^2$ consisting of one vertex, one
edge and two faces as shown in \firef{f:2sphere}.
   \begin{figure}[ht]
   \fig{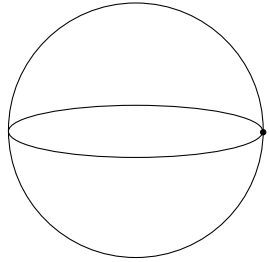}
  \caption{The polytope decomposition of $S^2$}\label{f:2sphere}
  \end{figure}
Using the fact that for $X_i$, $X_j$
simple $\Hom(X_i,X_j) = \delta_{ij}\kk$, it is easy to see that $H(S^2)
=\bigoplus_i \<X_i\>\otimes \<X_i^*\>= \kk$. It remains to show that
$A\colon  H(S^2) \rightarrow H(S^2)$ is the identity map or equivalently,
the induced map $H(S^2) \otimes H(S^2)^* \rightarrow \kk$ equals the
canonical pairing defined in \seref{s:prelim}. Consider the cylinder
$S^2\times I$ with cell decomposition as in \firef{f:saturn}. Note that
both boundary edges must be labeled by $\one$.
   \begin{figure}[ht]
   \fig{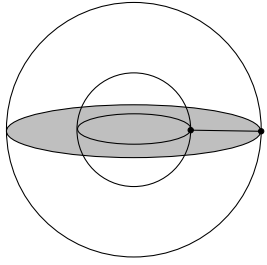}
  \caption{The cylinder over  $S^2$}\label{f:saturn}
  \end{figure}
The computation is then
straightforward: 
\begin{align*}
&\frac{1}{\DD^2}\sum_{X\in \Irr(\C)}d_X
\tzSphereProofI \cdot \tzSphereProofII 
=\frac{1}{\DD^2}\sum_{X\in \Irr(\C)}
\tzSphereProofIII \cdot \tzSphereProofIV \\
&\quad =\frac{1}{\DD^2}\sum_{X\in \Irr(\C)}d^2_X
 =1.
\end{align*}
The first equality follows from the
normalization of the pairing. The other two equalities are obvious. 
\end{example}

\section{Proof of independence of polytope
decomposition}\label{s:proof}

In this section, we give proofs of \thref{t:main1}, \thref{t:main2},
i.e. prove that TV invariants are independent of the choice of polytope
decomposition. The proof is based on \thref{t:moves},
\thref{t:moves_rel}, which state that any two decompositions can be
obtained from one another by a sequence of moves M1--M3 and their inverses.

First, we fix some notation. Unless otherwise stated, we denote simple
objects in $\C$ by 
$X_{i},X_{j}\dots$ and arbitrary objects by $A, B,\dots$. We let
$N_{1}^{i_{1}\dots i_{k}}=\dim(\<X_{i_{1}},\dots,X_{i_{k}}\>)$. 

We will now show that the TV state sum is invariant under M1--M3.

\subsection*{Invariance under M1}
First we consider move M1. Note that by applying M2 and M3, we can
transform an open book  with any number of pages to one with only one page
(see \firef{f:book}). 
\begin{figure}[ht]
\begin{align*}&\figscale{0.7}{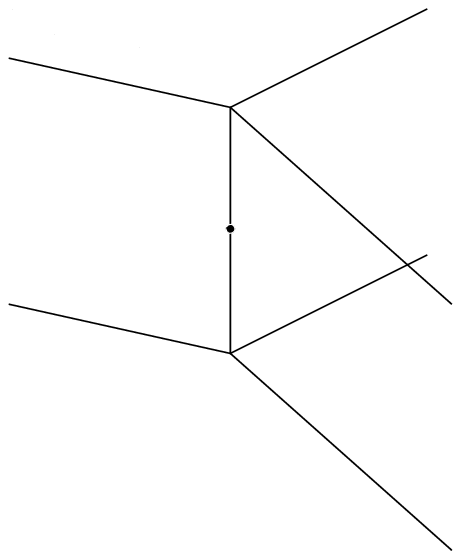}
    \xxto{\text{add edge to each page}}
    \figscale{0.7}{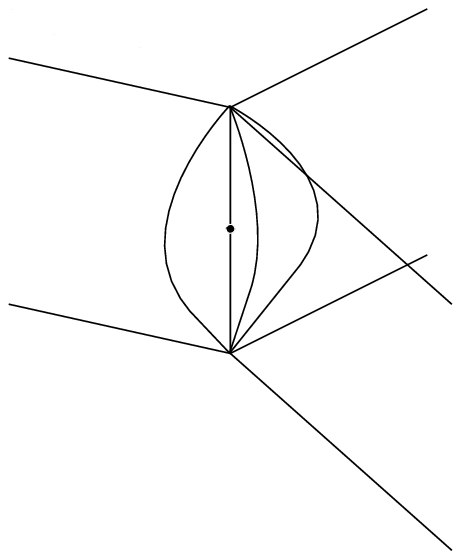}
    \\
     & \xxto{\text{add faces between new edges}}
    \figscale{0.7}{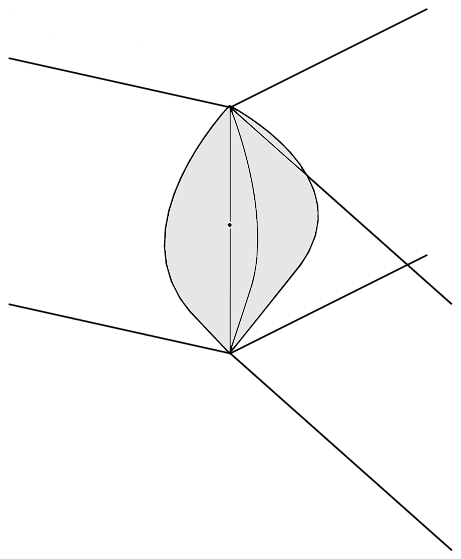}
     \xxto{\text{remove all pages but one}}  
   \figscale{0.7}{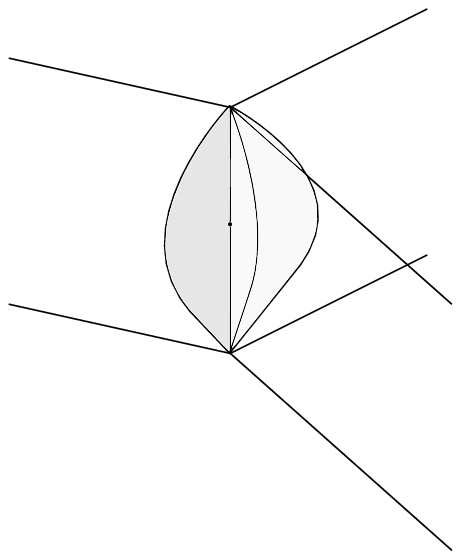} 
   \end{align*}
 \caption{Decomposing an open book into a single page book}\label{f:book}
 \end{figure} 
Thus, it suffices to prove invariance under M1 in this
special case. Drawing the dual graph in the vicinity of the vertex,
invariance under M1 is equivalent to the following equality:
$$
\frac{1}{\DD^2}\sum_{j,k\in \Irr(\C)}d_{j}d_{k}\tzPairingVII
=
\sum_{i\in \Irr(\C)}d_{i}\tzPairingIII
$$
Note the normalizing factor
$\frac{1}{\mathcal{D}^2}$ which comes from the fact that we are removing a
vertex.

Using semisimplicity of $\C$, it is easy to see that it suffices to show
this equality in the special case when $V=V_1\otimes \dots \otimes V_n$ is
simple:
$$
\frac{1}{\DD^2}\sum_{j,k\in \Irr(\C)}d_{j}d_{k}\tzPairingVIII
=
\sum_{i\in \Irr(\C)}d_{i}\tzPairingIX
$$
By \leref{l:pairing}, the right-hand side is equal to $\coev_V\colon \one
\to V\otimes V^*$. Since $\Hom(\one, V\otimes V^*)$ is one-dimensional,
the left-hand side is also a multiple of $\coev_V$. Composing it with the
evaluation morphism  $\ev_V$, we get
\begin{align*}
&\frac{1}{\DD^2}\sum_{j,k}d_jd_kN_{1}^{Vjk} =
\frac{1}{\DD^2}\sum_{j,k}N_{k^*}^{Vj}d_kd_j\\
&\quad=\frac{1}{\DD^2}\sum_{j}\Bigl(\sum_{k}N_{k^*}^{Vj}d_k\Bigr)d_j=
\frac{1}{\DD^2}\sum_{j}(d_{V}d_j)d_j=d_V,
\end{align*}
which proves that the left-hand side is equal to $\coev_V$.

\subsection*{Invariance under M2}
The invariance under M2 is seen as follows. By definition, the edge being
removed is incident to exactly two faces $c_1, c_2$. Each face bounds the
same two 3-cells $F_1, F_2$. In  \firef{f:move2}, we draw the dual graphs.
In each of the summands we have two graphs corresponding to cells $F_1,
F_2$, separated by a dot.  The equality follows
immediately from the fact that if $\ph_\al,\ph^\al$ and $\psi_\be,\psi^\be$ 
are dual bases, then so are $\ph_\al\ccc{X_i}\psi_\be$, $\psi^\be\ccc{X_i^*}\ph^\al$ 
(cf. \coref{c:composition2}). 
\begin{figure}[ht]
\begin{align*}
&\sum_{i,\alpha,\beta}d_{i}\tzMoveMtwoI \tzMoveMtwoII\\ 
&\quad= \sum_{i,\alpha,\beta} \tzMoveMtwoIII \tzMoveMtwoIV
\end{align*}
\caption{}\label{f:move2}
\end{figure}

\subsection*{Invariance under M3}
Finally, we consider M3. In this case the invaraince immediately follows
from \leref{l:pairing2}, whith  two subgraphs corresponding to two 3-cells
separated by the 2-cell being removed.

\section{Surfaces with boundary}\label{s:boundary}

In this section we extend the definition of TV TQFT to surfaces with
boundary (and 3-manifolds with corners). Recall that according to general
ideas of extended field theory (see \ocite{lurie}), an extended 3d 
TQFT should assign to a closed 1-manifold a 2-vector space, or an abelian
category, and to a 2-cobordism between two 1-manifolds, a functor between
corresponding categories (which in the special case of cobordism between
two empty 1-manifolds gives a functor $\Vect\to \Vect$, i.e. a  vector
space). In this section we show that the extension of the  TV TQFT to
1-manifolds assigns to a circle $S^1$ the category $Z(\C)$---the Drinfeld
center of the original spherical category $\C$. This result was
proved by Turaev in the special case when the original category $\C$ is
ribbon (see \ocite{turaev}); the general case has remained a
conjecture.

For technical reasons, it is more convenient to replace surfaces with
boundaries by surfaces with embedded disks. These two notions give
equivalent theories: given a surface with boundary, we can glue a disk to
every boundary circle and get a surface with embedded disks; conversely,
given a surface with embedded disks, one can remove the disks to get  a
surface with boundary. Moreover, in order to accommodate real-life
examples, we need to consider framing. This leads to the following
definition.

We denote
$$
D^2=[0,1]\times[0,1]
$$
and will call it {\em the standard disk} (it is, of course, a square, but
this is what a disk looks like in PL setting). We will also the marked
point $P_0$ on the boundary of $D^2$
$$
P_0=(0,1)\in \del D^2
$$
%

\begin{definition}\label{d:ex_surface}
A framed embedded disk  $D$ in a PL surface $N$ is the image of a PL
map
$$
\ph\colon D^2\to N
$$
which is a homeomorphism with the image, together with the point
$P=\ph(P_0)\subset \del D$.

An {\em extended surface} is a PL surface $N$ together with a finite
collection of disjoint framed embedded disks (see \firef{f:ex_surface}).
We will denote the set of embedded disks by $D(N)$.

A {\em coloring} of an extended surface is a choice of an object $Y_\al\in
\Obj Z(\C)$ for every embedded disk $D_\al$.
\end{definition}
\begin{figure}[ht]
\fig{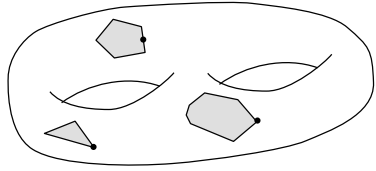}
\caption{Extended surface}\label{f:ex_surface}
\end{figure}
%

Next, we can define cobordisms between such surfaces. As usual, such a
cobordism will be a 3-manifold with boundary together with some ``tubes''
inside which connect the embedded disks on the boundary of $M$. The
following gives a precise definition in the PL category.

\begin{definition}\label{d:tube}
Let $M$ be a PL 3-manifold with boundary.

An open {\em embedded tube} $T\subset M$ is the image of a PL map
$$
\ph\colon [0,1]\times D^2\to M
$$
which is satisfies the conditions below, together with the oriented arc
$\ga=\ph([0,1]\times \{P_0\})$ (which we will call the {\em longitude}).

The map $\ph$ should satisfy:
\begin{enumerate}
  \item $\ph$ is a homeomorphism onto its image
  \item $T\cap \del M=\ph(\{0\}\times D^2) \cup \ph(\{1\} \times D^2)$
\end{enumerate}

We will call the disks $B_0=\ph(\{0\}\times D^2)$ and $B_1=\ph(\{1\}\times
D^2)$ the {\em bottom } and {\em top}  disks of the tube.

A closed embedded tube $T\subset M$ is the image of a PL map
$$
\ph\colon S^1\times D^2\to M
$$
which is satisfies the conditions below, together with the oriented arc
$\ga=\ph([0,1]\times \{P_0\})$ (the {\em longitude}) and the disk
$B=\ph(\{0\}\times D^2)\subset T$.

The map $\ph$ should satisfy:
\begin{enumerate}
  \item $\ph$ is a homeomorphism onto its image
  \item $T\cap \del M=\varnothing$
\end{enumerate}

\end{definition}

The longitude  $\ga$ determines the framing of the tube; the disk $B$ is
convenient for technical reasons; later we will get rid of it.

\begin{definition}\label{d:ex_manifold}
  An extended 3-manifold $M$ is an oriented  PL 3-manifold with boundary
  together with a finite collection of disjoint framed tubes $T_i\subset
  M$. We denote the set of tubes of $M$ by $T(M)$.

  A coloring of an extended 3-manifold $M$ is a choice of an object
    $Y_\al\in \Obj Z(\C)$ for every tube $T_\al$.
\end{definition}

\begin{figure}[ht]

\fig{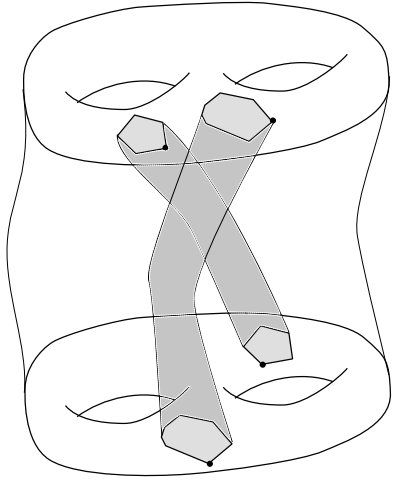}
\caption{Extended 3-manifold}\label{f:ex_manifold}
\end{figure}

Note that if $M$ is an extended 3-manifold, then its boundary $\del M$ has
a natural structure of an extended surface: the embedded disks are the
bottom and top disks  of the open tubes, and the marked points on the
boundary of embedded disks are the endpoints of the longitude arcs
$\ga_\al$, where $\al$ runs over the set of all open tubes in $M$.
Moreover, a coloring of $M$ defines a coloring of $\del M$: if an open
tube $T_\al$ is colored with $Y_\al\in \Obj Z(\C)$, we color the embedded
disk $\ph_\al(\{1\} \times D^2)$ with $Y_\al$ and the embedded
disk $\ph_\al(\{0\} \times D^2)$ with $Y_\al^*$.

Our main goal will be extending the TV invariants to such extended
surfaces and cobordisms. Namely, we will
\begin{enumerate}
  \item Define, for every colored extended surface $N$, the space
  $Z_{TV}(N, \{Y_\al\})$ which
  \begin{itemize}
   \item functorially  depends on colors $Y_\al$
  \item is functorial under homeomorphisms of extended surfaces
  \item has natural isomorphisms $Z_{TV}(\ov{N}, \{Y_\al^*\})=Z_{TV}(N,
\{Y_\al\})^*$
  \item satisfies the gluing axiom for surfaces
  \end{itemize}
 \item Define, for every colored extended 3-manifold $M$, a vector
    $Z_{TV}(M)\in Z_{TV}(\del M)$ (or, equivalently, for any
  colored extended 3-cobordism $M$ between colored extended surfaces $N_1,
  N_2$, a linear map $Z_{TV}(M)\colon Z_{TV}(N_1)\to Z_{TV}(N_2)$) so that
  this satisfies the gluing axiom for extended 3-manifolds.
\end{enumerate}

In the subsequent papers we will show that this  extended theory
actually coincides with the Reshetikhin--Turaev theory for the
modular category $Z(\C)$:
$$
Z_{RT, Z(\C)}=Z_{TV, \C}.
$$

The construction of the theory proceeds similar to the construction of TV
invariants. Namely, we will first define $Z_{TV}(N)$, $Z_{TV}(M)$ for
manifolds with a polytope decomposition and then show that the so defined
objects are independent of the choice of a polytope decomposition and
thus define an invariant of extended manifolds. 

\section{Extended combinatorial surfaces}\label{s:ext_surfaces}
We begin by generalizing the definition of a polytope decomposition to
extended surfaces.

\begin{definition}\label{d:comb_ex_surface}
    A  combinatorial extended surface $\N$ is a an extended surface $N$
  together with a  polytope decomposition such that
\begin{enumerate}
  \item The interior of each embedded disk is one of the 2-cells of the
    polytope decomposition.
  \item Each marked point $P_\al$ on the boundary of an embedded disk is a
    vertex (0-cell) of the polytope decomposition. 
\end{enumerate}
\end{definition}

We can now define the state space for such a surface. Let $\N$ be
a combinatorial extended surface, and $Y_\al, \al\in D(N)$, --- a
coloring of $\N$. Let $l$ be a labeling of edges of $\N$. Then we define
the state space
$$
H(\N, \{Y_\al\},l)=\bigotimes_C  H(C,l)
$$
where the product is over all 2-cells of $\N$ (including the embedded
disks) and
$$
H(C,l)=\begin{cases}
\<Y_\al,l(e_1), l(e_2),\dots, l(e_n)\>&\quad
                  C=D_\al \text{ -- an embedded disk}\\
\<l(e_1), l(e_2),\dots, l(e_n)\>&\quad C \text{ -- an ordinary 2-cell of
$\N$}
\end{cases}
$$
where $e_1,e_2,\dots$ are edges of $C$ traveled counterclockwise;
for the embedded disks, we also require that we start with the
marked point $P_\al$; for ordinary 2-cells of $\N$ the choice of starting
point is not important.

As usual, we now define
\begin{equation}\label{e:ex_state_sum}
  H(\N, \{Y_\al\})=\bigoplus_{l}H(\N, \{Y_\al\},l)
\end{equation}
where the sum is taken over all equivalence classes of simple labellings
$l$ of edges of $\N$. 

Note that so defined state space is functorial in $Y_\al$ and functorial
under homeomorphism of extended surfaces;  it is
 also immediate from the definition that one has a canonical isomorphism
$$
H(\ov{\N},Y^*_\al)=H(\N,Y_\al)^*.
$$

\begin{example}\label{x:n-sphere}
Let $\N$ be the sphere with $n$ embedded disks and the cell decomposition
shown in \firef{f:n-sphere}. 
\begin{figure}[ht]

\fig{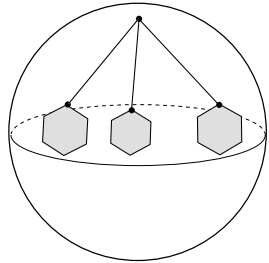}
\caption{$n$-punctured sphere}\label{f:n-sphere}
\end{figure}
Then
\begin{align*}
&H(\N,Y_1, \dots, Y_n)\\
&\qquad=\bigoplus_{X_1,\dots, X_n, U_1,
\dots, U_n\in \Irr(\C)}
\<X_1,U_1,X_1^*,\dots, X_n, U_n, X_n^*\>\otimes
\<U_1^*,Y_1\>\otimes\dots\otimes \<U_n^*,Y_n\>\\
&\qquad\simeq\bigoplus_{X_1,\dots, X_n\in \Irr(\C)}
\<X_1,Y_1,X_1^*,\dots, X_n, Y_n, X_n^*\>.
\end{align*}
where the last isomorphism is given by direct sum of rescaled compositions
\eqref{e:composition2}.
\end{example}

The first main result of this paper is the gluing axiom for the so defined
state space.

\begin{theorem}\label{t:gluing_surfaces}
Let $\N$ be a combinatorial extended surface and $D_a, D_b$ --- two
distinct embedded disks.  Let $\N'$ be the extended
surface obtained by removing the disks $D_a$, $D_b$ and connecting
the resulting boundary circles with a cylinder with the polytope
decomposition consisting of a single 2-cell and a single 1-cell as shown
below:
\begin{figure}[ht]
\fig{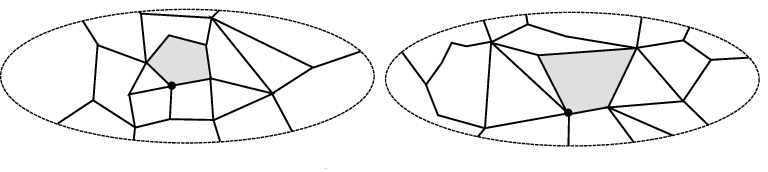}

\fig{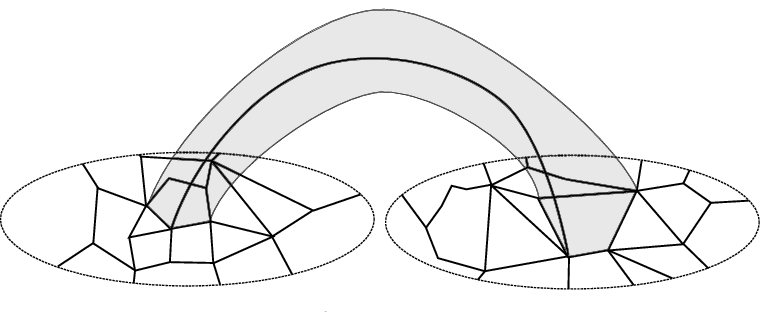}

\caption{Gluing of extended surfaces. To help
visualize the cylinder, it is colored light gray.}\label{f:gluing_surfaces}
\end{figure}

Thus, the set $D'$ of embedded disks of $\N'$ is
$D'=D(\N)\setminus\{a,b\}$

Then one has a natural isomorphism
$$
H(\N', \{Y_\al\}_{\al\in D'})=
\bigoplus_{Z\in \Irr(Z(\C))} H(\N, \{Y_\al\}_{\al\in D'},
Z, Z^*)
$$
where objects $Z,Z^*$ are assigned to embedded disks $D_a,D_b$.
\end{theorem}
\begin{proof}
For a given labeling $l$ of edges of $\N$, let
$$
H_0(l)=\bigotimes_C H (C,l)
$$
where the product is taken over all 2-cells of $\N$ (including the
embedded disks) except $D_a$, $D_b$. 
Then
$$
H(\N,\{Y_\al\}, Z, Z^*, l)=H_0(l)\otimes 
 \<Z, A\>\otimes \<Z^*, B\>
$$
where $A=l(e_1)\otimes l(e_2)\dots\otimes  l(e_n)$, 
where $e_1,e_2,\dots$ are edges of $D_a$ traveled counterclockwise
starting with the marked point $P_a$, and similarly for $B$.

On the other hand, for a given labeling $l'$ of edges of $\N'$, we have
$$
H(\N', \{Y_\al\}, l')=H_0(l)\otimes \<A\otimes l(e) \otimes B\otimes
l(e)^*\>
$$
where $l$ is the restriction of labeling $l'$ to edges of $\N$, and $e$ is
the added  edge connecting marked points $P_a$, $P_b$.

Thus, the theorem immediately follows from the following lemma.

\begin{lemma}\label{l:gluing_isom2} 
For any $A,B\in \Obj \C$, the map
\begin{equation}\label{e:gluing_isom2}
\begin{aligned}
\bigoplus_{Z\in \Irr(Z(\C))} \<Z, A\>\otimes \<Z^*, B\>&\to
  \bigoplus_{X\in \Irr(\C)}\<A, X , B, X^*\>\\
\ph\otimes \psi&\mapsto\bigoplus_{X\in \Irr(\C)}
\frac{\sqrt{d_X}\sqrt{d_Z}}{\DD} \quad \tzGluingAxiomI
\end{aligned}
\end{equation}
is an an isomorphism.
\end{lemma}
(The factor $\sqrt{d_X}\sqrt{d_Z}/\DD$ is introduced to make this isomorphism
agree with pairing \eqref{e:pairing}.)
\begin{proof}
By \thref{t:I}, we have
\begin{align*}
&\bigoplus_{X\in \Irr(\C)}\<A\otimes X \otimes B\otimes X^*\>
=\bigoplus_{X}\Hom_\C(A^*, X \otimes B\otimes X^*)
\\
&\quad=\Hom_\C(A^*, FI(B))=\Hom_{Z(\C)}(I(A^*), I(B))
\end{align*}
On the other hand, 
\begin{align*}
&\bigoplus_{Z\in \Irr(Z(\C))}
 \<Z, A\>\otimes \<Z^*, B\>=
\bigoplus_Z\Hom_\C(Z^*,A)\otimes \Hom_\C(Z,B)\\
&\quad
=\bigoplus_Z\Hom_{Z(\C)}(Z^*,I(A))\otimes \Hom_{Z(\C)}(Z,I(B))\\
&\quad=\bigoplus_Z\Hom_{Z(\C)}(I(A)^*, Z)\otimes
\Hom_{Z(\C)}(Z,I(B))\\
&\quad =\Hom_{Z(\C)}(I(A)^*,I(B))
\end{align*}
(using semisimplicity of $Z(\C)$). 
\end{proof}
This completes the proof of the lemma
and thus the theorem.
\end{proof}

\section{Invariants of  extended  3-manifolds}\label{s:ext_mfds}
We begin by generalizing the definition of a polytope decomposition to
extended 3-manifolds as defined in \deref{d:ex_manifold}.

\begin{definition}\label{d:comb_ex_manifold}
 A combinatorial extended 3-manifold $\M$ is  an extended PL 3-manifold
  with a polytope decomposition such that

  \begin{itemize}
   \item For an open tube $T_\al$, its interior is a single 3-cell of the
    decomposition. Moreover, the interior of the  ``bottom disk''
    $B_0=\ph_\al(\{0\}\times D^2)$ is a single 2-cell of the decomposition,
    and the marked point $P$ on the boundary of the bottom disk is a
    vertex of the decomposition, and similarly for the top disk
    $B_1=\ph_\al(\{1\}\times D^2)$.

   \item For a closed tube $T_\al$, the interior of the  disk
    $B_\al=\ph_\al(\{0\}\times D^2)$ is a single 2-cell of the
    decomposition, the marked point $P_\al\in \del B_\al$ is a vertex of
    the decomposition, and the complement $\Int(T_\al)-B_\al$ is
    a single 3-cell of the decomposition.

  \end{itemize}
\end{definition}

Note that this implies that the restriction of such a polytope
decomposition to the boundary of $\del M$ satisfies the conditions of
\deref{d:comb_ex_surface} and thus defines on $\del M$ the structure of a
combinatorial extended surface. It also this implies that $\M$
contains two kinds of 3-cells: usual cells (which are not contained in any
tube) and ``tube cells'', i.e. cells contained in one of the tubes. The
boundary of a usual 3-cell is a union of usual 2-cells; the boundary of a
 3-cell corresponding to an open  tube contains usual 2-cells and two
embedded disks;
the boundary of a  3-cell corresponding to a closed tube contains usual
2-cells and two copies of the disk $B_\al$ with opposite orientation. 

Finally, note that we have imposed no restriction on the longitude of the
tube: it is allowed (and usually will) intersect the edges of the
decomposition of the boundary tubes.

The following theorem is an analog of \thref{t:moves_rel}.

\begin{theorem}\label{t:moves_ext}
 Let $M$ be an extended 3-manifold. Then any two polytope decompositions
$\M', \M''$ of $\M$ which satisfy the conditions of
\deref{d:comb_ex_manifold} and agree on $\del M$ can be obtained  from
each other by a sequence of moves M1---M3 and their inverses such that all
intermediate decompositions also satisfy the conditions of
\deref{d:comb_ex_manifold} and agree with $\M', \M''$ on $\del M$.
\end{theorem}
\begin{proof}
Let us consider the manifold $\tilde M$ obtained by removing from $M$ the
interior of every tube and also the interior of the embedded disks on the
boundary of $M$. Then $\tilde M$ is a manifold with boundary
$$
\del \tilde M=(\del M-\cup \Int(D_\al))\cup \del \tilde M_{free}
$$
where the ``free boundary'' $\del \tilde M_{free}$ is the union of side
surfaces $I\times \del D^2$ of the tubes (for closed tubes, $S^1\times
\del D^2$).

Obviously, polytope decompositions $\M', \M''$ satisfying the conditions
of the theorem determine decomposition of $\tilde M$ which agree on the
subset $X=(\del M-\cup \Int(D_\al))\subset \del \tilde M$. Now the result
follows from \thref{t:moves_rel2}.
\end{proof}

Recall that for usual oriented  3-cell $F$ and a choice of edge labeling
$l$, we have defined the vector $Z_{TV}(F,l)\in H(\del
F,l)$ defined by \eqref{e:Z(F)}. We can now generalize it to tube
cells. Namely, let $l$ be an edge coloring of an extended
combinatorial 3-manifold $\M$ and let $T_\al\subset M$ be an open tube,
with the longitude $\ga_\al$ and color $Y_\al\in Z(\C)$. Since $T$ is
homeomorphic to $[0,1]\times D^2\simeq D^3$--- a 3-ball, the boundary
$\del T$ is homeomorphic to $S^2$; thus, the polytope
 decomposition of $T$ defines a polytope decomposition of $S^2$. 

 Let $\Ga$ be  the dual graph of this cell decomposition. We can connect
the marked points on the top and bottom disks to the vertex of the dual
graph corresponding to these disks; together with  the longitude $\ga$,
this gives an oriented arc on the surface of the sphere whose endpoints
are two distinct vertices of $\Ga$. For every 2-cell
$C\in \del F$ (including the embedded disks), choose a vector $v_C\in
H (C,l)^*$. Thus,  we get a graph $\hat\Ga$ of the type considered in
\seref{s:center}, i.e.  colored graph $\Ga$ on the surface of the sphere
together with a colored framed arc inside as shown in \firef{f:dual_tube}. 
\begin{figure}[ht]
  \fig{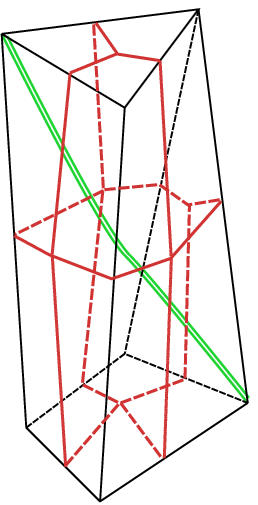}\qquad \fig{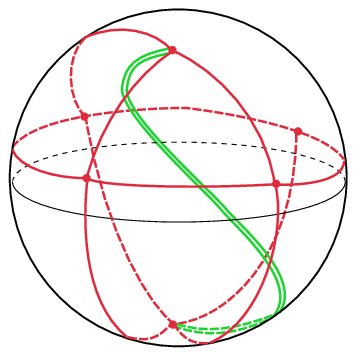}\qquad
\fig{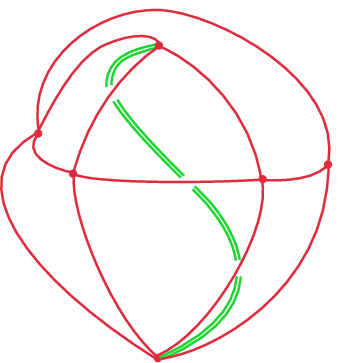}
  \caption{Dual graph for a tube cell. The longitude is shown by double
green line. }\label{f:dual_tube}
\end{figure}
By \thref{t:spherical2} this
defines a number $Z_{RT}(\hat\Ga)$; as before, we let
\begin{equation}\label{e:Z(F)_ex}
(Z(F,l), \otimes v_C)=Z_{RT}(\hat\Ga).
\end{equation}

In a similar way we define the invariant for closed tubes. 

We can now generalize the constructions of \seref{s:TV}  to
extended 3-manifolds.

\begin{definition}\label{d:TV_invariant_ex}
  Let $\M$ be an extended combinatorial 3-manifold with boundary and $\C$
-- a  spherical category. Then for any edge coloring $l$ and a
coloring $Y_\al$ of the tubes $T_\al\subset \M$, define the vector
$$
Z_{TV}(\M, \{Y_\al\}, l)\in  H(\del \M,\{Y_\al\},l)
$$
by 
$$
Z_{TV}(\M,l)=\ev\Bigl(\bigotimes_F Z(F,l)\Bigr)
$$
where 
\begin{itemize}
    \item $F$ runs over all 3-cells in $M$ (including the tube cells),  
    each taken with the induced orientation, so that  
    $$
      \bigotimes_F Z(F,l)\in \bigotimes_F H(\del F,l)=H(\del
\M,l)\otimes
      \bigotimes_{c}  H(c',l)\otimes H(c'',l)
    $$
    (compare with \leref{l:delF})  

  \item $c$ runs over all unoriented 2-cells in
    the interior of   $M$, including the disks $B_\al$ inside the closed
     tubes, and $c', c''$ are the two orientations of such a
    cell, so that   $c'=\ov{c''}$. 
  \item $\ev$ is the tensor product over all $c$ of evaluation maps
    $H(c',l)\otimes  H(c'',l)=H(c',l)\otimes H(c',l)^*\to \kk$
\end{itemize}

Finally, we define
\begin{equation}\label{e:Z_ext}
Z_{TV}(\M,\{Y_\al\})=\DD^{-2v(\M)}\sum_{l} \Bigl(
Z_{TV}(\M,\{Y_\al\},l)\prod_{e}d^{n_e}_{l(e)}\Bigr)
\end{equation}
where 
\begin{itemize}
\item the sum is taken over all equivalence classes of 
   simple labellings of $\M$,
\item $e$ runs over the set of all (unoriented) edges of $\M$
\item $\DD$ is the dimension of the category $\C$ (see \eqref{e:DD}), and
  $$
  v(\M)=\text{number of internal vertices of } \M+
   \frac{1}{2}\text{(number of vertices on }\del \M)
$$
\item $d_{l(e)}$ is the categorical dimension of $l(e)$ and 

$$n_e=\begin{cases}1, &\quad e \text{ is an internal edge}\\
                   \tfrac{1}{2}, &\quad e \in \del \M
                   \end{cases}
$$
\end{itemize}
\end{definition}
Note that in this definition, edges and vertices on the boundary of the
tubes are considered internal unless they are also on $\del \M$.

\begin{theorem}\label{t:main4}
\par\noindent
  \begin{enumerate} 
      \item $Z_{TV}(\M)$ satisfies the gluing axiom: if $\M$ is an
        extended combinatorial 3-manifold with boundary $\del \M=\N_0\cup
        \N\cup\ov{\N}$, and $\M'$ is the manifold obtained by identifying
        boundary components $\N,\ov{\N}$ of $\del \M$ with the obvious cell
        decomposition \textup{(}if $\N$  contains embedded disks,
          then we may need to erase them so that the interior of 
            resulting tubes have  exactly  one 3-cell\textup{)}, then we
        have
    $$
      Z_{TV}(\M')=\ev_{H(\N)}Z_{TV}(\M)=\sum_\al (Z_{TV}(\M),
          \ph_\al\otimes \ph^\al),
    $$
        where $\ev$ is the evaluation map $H(\N)\otimes H(\ov{\N})\to
        \kk$, and $\ph_\al\in H(\N)$,  $\ph^\al\in H(\ov{\N})$ are dual
        bases.

      \item If a $M$ is an extended  PL 3-manifold, and $\M',
          \M''$ are two polytope
        decompositions of $M$  which agree     on the boundary, then
        $Z(\M',\{Y_\al\})=Z(\M'',\{Y_\al\})$. 
        
    \item For a combinatorial 2-manifold $\N$, define $A\colon H(\N)\to
      H(\N)$ by 
      \begin{equation}\label{e:projector2}
       A=Z_{TV}(\N\times I)
      \end{equation}
       Then $A$ is a projector: $A^2=A$.

    \item For a combinatorial extended  2-manifold $\N$, define the vector
        space
     \begin{equation}\label{e:modular_functor2}
    Z_{TV}(\N)=\mathrm{Im}(A\colon H (\N)\to H (\N))
    \end{equation}
  where $A$ is the projector \eqref{e:projector2}.
  Then the space $Z_{RT}(\N)$ is an invariant of PL manifolds: if $\N',
  \N''$ are two different polytope decompositions of the same extended PL
    manifold  $N$, then one has a canonical isomorphism $Z(\N')\simeq
  Z(\N'')$.
\end{enumerate}
\end{theorem}
\begin{proof}
 The proof is parallel to the proof of \thref{t:main2}. The only
 new ingredient is in the proof of part (1), i.e. the gluing axiom for
  3-manifolds: if the component of boundary along which we are gluing
    contains embedded disks, we need to erase them so that in the
  resulting manifold, interior of each tube is exactly one 3-cell. Thus,
  we need to check that our that $Z(\M)$ is unchanged under this
  operation. The proof of this is similar to invariance under M3 move
  proved in \seref{s:proof}. Details are left to the reader. 

\end{proof}

Finally, we also note that our extended theory satisfies the gluing axiom
for extended surfaces. 
\begin{theorem}\label{t:gluing_surfaces2}
  Under the assumptions of \thref{t:gluing_surfaces}, one has a natural
  isomorphism
$$
Z(\N', \{Y_\al\}_{\al\in D'})=
\bigoplus_{Z\in \Irr(Z(\C))} Z(\N, \{Y_\al\}_{\al\in D'},
Z, Z^*)
$$
where objects $Z,Z^*$ are assigned to embedded disks $a,b$.
\end{theorem}
\begin{proof}
  Recall that by  \thref{t:gluing_surfaces}, one has an isomorphism
  $$
  G\colon \bigoplus_{Z\in \Irr(Z(\C))} H(\N, \{Y_\al\}_{\al\in D'},
Z, Z^*)\isoto H(\N', \{Y_\al\}_{\al\in D'})
  $$
   Since $Z(\N)$ is defined as the image of the projector $A\colon H(\N)\to
H(\N)$, and similarly for $Z(\N')$, it suffices to prove that the
following diagram is commutative: 
$$\begin{tikzpicture}
\matrix[column sep=1cm,row sep=20pt]
{ \node (m11)  {$H(\N, Z,Z^*)$}; &
 \node (m12)  {$H(\N')$};\\
 \node (m21)  {$H(\N, Z,Z^*)$}; &
 \node (m22)  {$H(\N')$};\\
};
\draw[->] (m11)--(m12) node[pos=0.5,above]{$G$};
\draw[->] (m21)--(m22) node[pos=0.5,above]{$G$};
\draw[->] (m11)--(m21) node[pos=0.5,left]{$A$};
\draw[->] (m12)--(m22) node[pos=0.5,right]{$A'$};
\end{tikzpicture}
$$
or equivalently, that for any $\ph\in H(\N, Z,Z^*)$, $\ph'\in H(\ov{\N},
Z,Z^*)$, we have 
$$
(Z(\N\times I,Z,Z^*), \ph\otimes \ph')=
(Z(\N'\times I), G(\ph)\otimes G(\ph')).
$$

Comparing both sides, we see that the only difference is that $\N\times I$ 
contains a pair of cylinders $D_a\times I, D_b\times I$:
$$
\tzGluingV
$$
whereas $\N'\times I$ contains instead a single  cell $C\times I$, where $C$ 
is the cylinder connecting boundary circles $\del D_a, \del D_b$:
$$
\tzGluingVI
$$

Thus, it suffices to prove that for any collection 
\begin{alignat*}{2}
&\ph_a\in H(D_a)=\<A,Z\>,\qquad && \ph_b\in H(D_b)=\<B,Z^*\>,\\
&\ph'_a\in H(D_a)^*=\<(A')^*,Z^*\>,\qquad &&\ph_b\in H(D_b)^*=\<(B')^*,Z\>,\\
&\psi_a\in H(\del D_a\times I)=\<X_k,A',X_k^*, A^*\>,\qquad &&
  \psi_b\in H(\del D_b\times I)=\<X_l,B',X_l^*, B^*\>
\end{alignat*}
we have 
\begin{align*}
&Z(D_a\times I, \ph_a\otimes\ph_a'\otimes \psi_a) 
\cdot  Z(D_b\times I, \ph_b\otimes\ph_b'\otimes \psi_b) \\
&\quad=\sum_{i,j}\sqrt{d_i}\sqrt{d_j}\ Z(C\times I, \psi_a,\psi_b, G (\ph_a\otimes \ph_b),
  G (\ph'_a\otimes \ph'_b))
\end{align*}
(the factors $\sqrt{d_i}$, $\sqrt{d_j}$ appear because $\N'\times I$ contains two 
extra edges on the boundary, labeled $i,j$.)

The left hand side is given by 
$$
LHS=\tzGluingI\cdot\qquad  \tzGluingII
$$
Combining explicit computation given in \seref{s:computations1} with the
formula for gluing in \leref{l:gluing_isom2}, we see that  the right hand
side is given by
$$
RHS=\sum_{i,j} \frac{d_i d_j d_Z}{\DD^2} \tzGluingIII
$$
Using \leref{l:pairing}, we can rewrite it as
$$
RHS=\sum_{j} \frac{d_j d_Z}{\DD^2}\tzGluingIV
$$
Since it follows from \leref{l:projector} that for any simple $Z\in \Obj Z(\C)$ and a morphism $\Ph\in \Hom_\C(Z,Z)$, we have 
$$
\frac{1}{\DD^2}\sum_j d_j \quad \tzProjectionIa = \frac{1}{d_Z} \tr (\Ph) \id_Z
$$
this easily implies that the LHS is equal to RHS.

\end{proof}

\begin{example}\label{x:Z(S)}
  Let $\N$ be the sphere with $n$ embedded disks,
colored by objects $Y_1,\dots, Y_n\in \Obj Z(\C)$ (see
\exref{x:n-sphere}). Then 
$$
Z(\N,Y_1, \dots, Y_n)=\Hom_{Z(\C)}(\one, Y_1\otimes
Y_2\otimes\dots\otimes Y_n).
$$
Indeed, by \exref{x:n-sphere}, we have 
$$
H(\N,Y_1, \dots, Y_n)=\bigoplus_{i_1,\dots, i_n\in \Irr(\C)}
\<X_{i_1}, Y_1, X_{i_1}^*,\dots, X_{i_n}, Y_n, X_{i_n}^*\>.
$$

By a direct computation done  in \seref{s:computations}, we see that the 
operator $A=Z(\N\times I)\colon H(\N)\to H(\N)$  is given by 
\begin{align*}
\ph \mapsto &\frac{1}{\DD^{2(n+1)}}\sum_{l, j_1,\dots, j_n\in \Irr(\C)} 
d_l \prod_{a=1}^n 
         \sqrt{d_{i_a}} \sqrt{d_{j_a}} \\
&\qquad \tzNsphereV
\end{align*}

Consider now the subspace $W\subset H(\N,Y_1, \dots, Y_n)$ spanned by
elements of the form
\begin{align*}
&\bigoplus_{j_1,\dots, j_n}\prod_{a=1}^n \sqrt{d_{j_a}} \quad 
\tzNsphereVI\\[10pt] 
&\quad  \psi\in \Hom_{Z(\C)}(\one,
Y_1\otimes\dots\otimes  Y_n)
\end{align*}
Clearly, $W\simeq \Hom_{Z(\C)}(\one, Y_1\otimes\dots\otimes  Y_n)$.

Now, it follows from the previous computation and \leref{l:projector}
that for any $\ph \in H(\N, Y_1,\dots, Y_n)$, we have $A\ph\in W$; on the
other hand, it is immediate that if $\psi\in W$,
then $A\psi=\psi$. Therefore, $A$ is the projector onto $W\simeq
\Hom_{Z(\C)}(\one, Y_1\otimes\dots\otimes  Y_n)$.

\end{example}

\section{Some computations}\label{s:computations}
In this section we give some explicit computations of the TV invariants.

\subsection{Cylinder over an annulus}\label{s:computations1}

Let $F=S^1\times I\times I$ be the cylinder over an annulus, as shown
below. 

\tzCylinderI
 
Then 
$$
\del F=C_a\cup C_b \cup C_{in}\cup C_{out} \cup C\cup \ov{C}
$$
where $C_a, C_b$ are the left and right annuli, $C_{in}$ and $C_{out}$ are
the inner and outer cylinders, and $C$ is the internal cell:
\begin{align*}
&C_a=\tzCL
\qquad
C_b=\tzCR\\
&C_{in}=\tzCin
\qquad
C_{out}=\tzCout
\qquad
C=\tzC
\end{align*}

The pullback of the cell decomposition of $\del F\simeq S^2$ to the sphere
is homeomorphic to the cube shown below:
$$
\tzCube
$$

Drawing the dual graph, we see that given a collection 
\begin{align*}
 &\psi_L\in H(\ov{C_a})=\<A,X_k, (A')^*, X_k^*\>,\qquad 
 \psi_R\in H(\ov{C_b})=\<B_1^*,X_l, B', X_l^*\>\\
 &\psi_{in} \in H(\ov{C_{in}})=\<A^*,X_i, B, X_i^*\>,\qquad 
 \psi_{out}\in H(\ov{C_{out}})=\<A',X_j, (B')^*, X_j^*\>\\
 &\ph\in H(C)=\<X_l,X_j^*,X_k^*,X_j\>,\qquad 
 \ph'\in H(\ov{C})=\<X_j^*,X_k,X_j,X_l^*\>,
\end{align*}
the value $(Z(F),\psi_L\otimes \psi_R\otimes
\psi_{in}\otimes\psi_{out}\otimes\ph\otimes\ph')$ is given by
the following graph:
$$
\tzDualGraphCylII
$$

\subsection{Sphere with $n$ holes}\label{s:n-sphere}

  Let $\N$ be the sphere with $n$ embedded disks,
colored by objects $Y_1,\dots, Y_n\in \Obj Z(\C)$ (see
\exref{x:n-sphere}). Choose the cell decomposition of $N$  as in 
\exref{x:n-sphere}; then 
\begin{align*}
&H(\N,Y_1, \dots, Y_n)\\
&\qquad=\bigoplus_{X_1,\dots, X_n, U_1,
\dots, U_n\in \Irr(\C)}
\<X_1,U_1,X_1^*,\dots, X_n, U_n, X_n^*\>\otimes
\<U_1^*,Y_1\>\otimes\dots\otimes \<U_n^*,Y_n\>\\
&\qquad\simeq\bigoplus_{X_1,\dots, X_n\in \Irr(\C)}
\<X_1,Y_1,X_1^*,\dots, X_n, Y_n, X_n^*\>.
\end{align*}

Consider now the cylinder $\N\times I$ with the cell decomposition shown
in \firef{f:puncturedsphere_cylinder}. 

\begin{figure}[ht]
 \figscale{0.7}{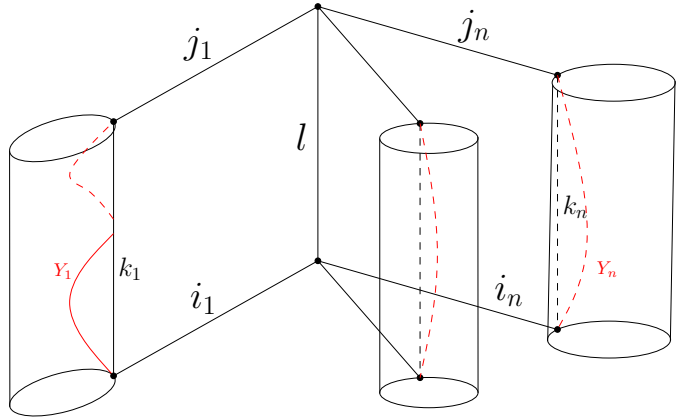}
 \caption{Cylinder over sphere with $n$ embedded
  disks}\label{f:puncturedsphere_cylinder}
\end{figure}

This cell decomposition contains $(n+1)$ 3-cells: $n$ open tubes and one
large 3-cell. Thus, the invariant $Z(\N\times I)$ is given by 

$$
(Z(\N\times I), \ph\otimes \ph')=
\sum_{l,  k_1, \dots, k_n\in \Irr(\C)} 
D \cdot  Z_0\cdot \dots\cdot Z_n
$$
where 
\begin{align*}
\ph&=\ph_0\otimes \ph_1\otimes \dots \otimes \ph_n 
      \in H(\N, i_1,\dots, i_n)\\
&\qquad=
  \<X_{i_1},U_1,X_{i_1}^*,\dots, X_{i_n}, U_n, X_{i_n}^*\>\otimes
  \<U_1^*,Y_1\>\otimes\dots\otimes \<U_n^*,Y_n\>\\
\ph'&=\ph'_0\otimes \ph'_1\otimes \dots \otimes \ph'_n
   \in H(\N, j_1,\dots, j_n)^*\\
   &\qquad=\<X_{j_n},V^*_n,X_{j_n}^*,\dots, X_{j_1}, V^*_1, X_{j_1}^*\>
   \otimes \<V_1,Y_1^*\>\otimes\dots\otimes \<V_n,Y_n^*\>
\end{align*}
$D$ is the normalization factor:
$$
D=\frac{1}{\DD^{2(n+1)}}d_l  
\prod_{a=1}^n d_{k_a}
         \sqrt{d_{i_a}} \sqrt{d_{j_a}} \sqrt{d_{U_a}} \sqrt{d_{V_a}} 
$$
and $Z_0, \dots, Z_n$ are the factors corresponding to the $(n+1)$ 3-cells
of the decomposition:
\begin{align*}
Z_0&=\quad \tzNsphereI\\
Z_i&=\quad \tzNsphereII
\end{align*}
Using \leref{l:pairing2}, we see that it can be rewritten as follows: 
\begin{align*}
(Z(\N\times I), \ph\otimes \ph')&=\\
\sum_{l,  k_1, \dots, k_n\in \Irr(\C)}
&D\qquad \tzNsphereIII
\end{align*}
Thus, identifying $H(\N,i_1,\dots, i_n)\simeq 
\<X_{i_1}, Y_1, X_{i_1}^*,\dots)$ as in \exref{x:n-sphere} and
using \leref{l:pairing}, we see that 
\begin{align*}
(Z(\N\times I), \ph\otimes \ph')&=
\sum_{l\in \Irr(\C)}
\frac{1}{\DD^{2(n+1)}}d_l  
\prod_{a=1}^n 
         \sqrt{d_{i_a}} \sqrt{d_{j_a}} \\
& \tzNsphereIV
\end{align*}



\end{document}